\documentclass[10pt,a4paper,twoside]{amsart}

\usepackage{amsmath, amssymb, amsthm}
\usepackage{amsaddr}
\usepackage{siunitx}
\usepackage[utf8]{inputenc}
\usepackage[T1]{fontenc}
\usepackage{lmodern}
\usepackage{chngcntr}
\usepackage{geometry}
\usepackage{graphicx}
\usepackage{subcaption}
\usepackage{multirow}
\graphicspath{{./figs/}}
\usepackage[normalem]{ulem}

\def\R{\mathbb{R}}

\def\N{\mathbb{N}}

\def\cA{\mathcal{A}}

\def\cF{\mathcal{F}}

\def\cI{\mathcal{I}}

\def\cS{\mathcal{S}}

\def\a{\alpha}

\def\d{\delta}
\def\D{\Delta}

\def\o{\omega}
\def\oo{\overline{\omega}}
\def\G{\Gamma}
\def\GD{{\Gamma_{\rm D}}}

\def\p{\partial}

\def\Th{\mathcal{T}_h}
\def\Tl{\mathcal{T}_\ell}
\def\Thsym{\mathcal{T}_h^\mathrm{sym}}
\def\Tlsym{\mathcal{T}_\ell^\mathrm{sym}}
\def\Nh{\mathcal{N}_h}
\def\Eh{\mathcal{E}_h}

\def\Ih{\mathcal{I}_h}
\def\hIh{\widehat{\mathcal{I}}_h}

\def\Pred{P_3^\mathrm{red}}
\def\Sdkt{\cS^\mathrm{dkt}}
\def\Idkt{\cI^\mathrm{dkt}_h}

\def\tE{\widetilde{E}}
\def\tEh{\widetilde{E}_h}
\def\tEpenh{\widetilde{E}_{\mathrm{pen},h}}

\def\DD{{\rm D}}

\def\wto{\rightharpoonup}

\def\estop{\varepsilon_\mathrm{stop}}

\newcommand{\dv}[1]{\,{\mathrm d}#1}

\DeclareMathOperator{\diver}{div}

\let\oldmarginpar\marginpar
\renewcommand\marginpar[1]{
  \oldmarginpar[\raggedleft\footnotesize #1]
  {\raggedright\footnotesize #1}}

\theoremstyle{definition}
\newtheorem{definition}{Definition}
\newtheorem{remark}[definition]{Remark}

\theoremstyle{plain}

\newtheorem{proposition}[definition]{Proposition}
\newtheorem{theorem}[definition]{Theorem}
\newtheorem{corollary}[definition]{Corollary}
\newtheorem{algorithm}[definition]{Algorithm}
\numberwithin{definition}{section}
\numberwithin{equation}{section}

\begin{document}
\title[Simulating Constrained Bilayer Plate Bending]{Stable Gradient Flow Discretizations for Simulating Bilayer Plate Bending with Isometry and Obstacle Constraints}
\author{Sören Bartels, Christian Palus}
\address{Department of Applied Mathematics \\
University of Freiburg}
\date{\today}
\begin{abstract}
  Bilayer plates are compound materials that exhibit large bending deformations when exposed to environmental changes that lead to different mechanical responses in the involved materials. In this article a new numerical method which is suitable for simulating the isometric deformation induced by a given material mismatch in a bilayer plate is discussed. A dimensionally reduced formulation of the bending energy is discretized generically in an abstract setting and specified for discrete Kirchhoff triangles; convergence towards the continuous formulation is proved. A practical semi-implicit discrete gradient flow employing a linearization of the isometry constraint is proposed as an iterative method for the minimization of the bending energy; stability and a bound on the violation of the isometry constraint are proved. The incorporation of obstacles is discussed and the practical performance of the method is illustrated with numerical experiments involving the simulation of large bending deformations and investigation of contact phenomena.
\end{abstract}
\maketitle

\section{Introduction}

\subsection{Scope} Modern nano-scale applications motivate the development of mathematical methods that are suitable for simulating the bending of bilayer plates (see~\cite{BBMN18}). We investigate the numerical treatment of a nonlinear two-dimensional bilayer plate bending model. Our results improve the methods described in~\cite{bartels2017} in several ways: (1) We employ a general finite element space based on triangular elements instead of quadrilaterals and thus avoid restrictive mesh conditions; (2) our scheme is fully practical as it only requires the solution of one linear system in every (pseudo-)time step of the discrete gradient flow  instead of a non-convex minimization problem; (3) we discuss the iterative treatment of an obstacle constraint in the minimization problem. 
We present our method in an abstract framework applicable to general finite element methods. The practical performance is illustrated with numerical experiments.

\subsection{Bilayer plate model} Bilayer plates are made from compound materials consisting of two layers with slightly different mechanical properties. An external influence, e.\,g. a temperature change, might cause one of the layers to contract while the other one expands, which leads to a \emph{material mismatch} and thus induces a bending deformation of the plate. We consider an extended Kirchhoff plate model which only allows for pure bending deformations of an initially flat two-dimensional plate, i.\,e. no stretching or shearing is supposed to occur in the deformed configuration of the plate. The use of this dimensionally reduced plate model has been rigorously justified in~\cite{schmidt2007, schmidt2007b} extending results from~\cite{friesecke2002, friesecke2002b}. Given a domain $\o \subset \R^2$ describing the reference configuration of the bilayer plate and some parameter $\a \in \R$, we seek minimizing deformations $y \colon \o \to \R^3$ for the elastic energy
\begin{equation}\label{eq:energy}
  E[y] = \frac 1 2 \int_\o |H(y) -\a I_2|^2 \dv{x} - \int_\o f \cdot y \dv{x},
\end{equation}
subject to the isometry constraint 
\begin{equation}\label{eq:iso-const}
  [\nabla y]^\top \nabla y = I_2,
\end{equation}
which reflects the pure bending condition mentioned above. Throughout this article, $I_2$ denotes the $2$-dimensional identity matrix and $H(y)$ the second fundamental form of the parametrized surface given by the deformation~$y$, i.\,e. $H_{ij}(y) = \nu \cdot \p_i\p_j y$ with the unit normal $\nu = \p_1 y \times \p_2 y$. 
The parameter $\a$ represents a homogeneous and isotropic material mismatch between the two layers.
We prescribe the boundary conditions $y = y_\DD$ and $\nabla y = \phi_\DD$ on a subset~$\GD$ of the boundary~$\p\o$ with positive one-dimensional length. The boundary data are assumed to be compatible with the density result in~\cite{hornung} in the sense that there exists an isometry $\widetilde{y}_\DD \in H^2(\o)^3$, such that $y_\DD = \widetilde{y}_\DD|_\GD$ as well as $\phi_\DD = \nabla \widetilde{y}_\DD|_\GD$, and that every isometry with this property can be approximated in $H^2(\o)^3$  by smooth isometries with the same values on $\GD$. 

The existence of minimizers for the constrained problem follows from the direct method in the calculus of variations.
The term $\a I_2$ in~\eqref{eq:energy} describes stress inherent in the plate when no outer body forces $f \colon \o \to \R^3$ are affecting it. 
In the absence of intrinsic strain (i.\,e. if $\alpha = 0$), the energy functional is quadratic and its numerical treatment was investigated in~\cite{bartels2013}. If $\alpha \neq 0$, a critical nonlinear term occurs and the numerical treatment becomes more delicate.
For any sufficiently smooth regular surface satisfying the isometry constraint~\eqref{eq:iso-const}, i.\,e. $\p_i y \cdot \p_j y = \d_{ij}$, we have for the Frobenius norm of the second fundamental form that
\[
|H|^2 = |D^2y|^2.
\]
This allows us to rewrite the energy $E[y]$ in~\eqref{eq:energy} as
\begin{equation}\label{eq:energy-re}
  \tE[y] = \frac 1 2 \int_\o |D^2 y|^2 \dv{x} - \a \int_\o \D y \cdot[\p_1 y \times \p_2 y] \dv{x} + \a^2 |\o| - \int_\o f \cdot y \dv{x}.
\end{equation}
Following~\cite{bartels2015}, we define the \emph{set $\cA$ of admissible deformations} as
\begin{equation}
  \cA = \left\{ y \in H^2(\o)^3 : y|_{\GD} = y_\DD,\> \nabla y|_{\GD} = \phi_\DD,\> [\nabla y]^\top \nabla y = I_2 \text{ a.\,e. in } \o \right\}
\end{equation}
and its \emph{tangent space at a point} $y \in H^2(\o)^3$ as
\begin{equation}
  \cF[y] = \left\{ w \in H^2(\o)^3 : w|_{\GD} = 0,\> \nabla w|_{\GD} = 0,\> [\nabla w]^\top \nabla y + [\nabla y]^\top \nabla w = 0 \text{ a.\,e. in } \o \right\}.
\end{equation}

We employ a spatial discretization based on discrete Kirchhoff triangles. The degrees of freedom for these finite elements are given by nodal function values and derivatives, cf.~\cite{batoz1980}.
These elements are easy to implement and the (linearized) isometry constraint can be explicitly enforced at the nodes of the triangulation. The minimization of the corresponding discrete energy functional is then realized via the semi-implicit discrete gradient flow defined by
\begin{align*}
\bigl(d_t y^k, v\bigr)_* = -&\bigl(D^2 y^k,D^2 v\bigr) + \bigl(f,v\bigr) + \a \bigl(\D v, \p_1 y^{k-1} \times \p_2 y^{k-1}\bigr) &\\
&+ \a \bigl(\D y^{k-1}, \p_1 v \times \p_2 y^{k-1}\bigr) + \a \bigl(\D y^{k-1}, \p_1 y^{k-1} \times \p_2 v\bigr)
\end{align*}
for all  $v\in\cF[y^{k-1}]$, where $d_t y^k = \tau^{-1}(y^k - y^{k-1})$ denotes the backward difference quotient with $\tau > 0$ and $(\cdot,\cdot)_*$ a scalar product on $\{w \in H^2(\o)^3 :  w|_{\GD} = 0,\> \nabla w|_{\GD} = 0\}$. Given an initial value $y^0 \in \cA$, we regard $d_t y^k$ as the unknown and find that there exists a unique solution $d_t y^k \in \cF[y^{k-1}]$ in every iteration step $k>0$.
This leads to a numerical scheme with guaranteed energy decay that converges to stationary configurations. We note that, in general, the iterates will not satisfy the isometry constraint exactly. Instead, the updates in every (pseudo-)time step are taken from the respective tangent spaces, which result from a linearization of the isometry constraint, and no projection of the new iterate onto the admissible set is included. Following the observations in~\cite{bartels2013,bartels2017}, we are able to prove that the isometry constraint is satisfied up to a small error which is independent of the number of iterations and controlled by the step size $\tau$.

 We then go on to discuss the inclusion of an obstacle constraint via a penalty method. This approach is numerically robust and easy to realize as it can be included in the discrete gradient flow without further difficulties. We note that recently an approach based on discontinuous Galerkin elements has been proposed for the discretization of plate beding models (cf.~\cite{BonNocNto19, BonNocNto20}).

\subsection{Outline} The outline of the article is as follows. In Section 2 we present an abstract framework for analyzing the convergence of a finite element that illustrates the key ingredients of our analysis. In Section 3 we collect some preliminary results about discrete Kirchhoff triangle elements, which we use in the discretization of the energy functional~\eqref{eq:energy-re}. The main results are contained in Sections 4, 5 and 6. In Section 4 we prove the $\G$-convergence of the discretized energy functionals to the continuous energy. In Section 5 we propose a new method for the numerical minimization of the discrete energies and prove its stability as well as an error bound for the constraint violation. Section 6 discusses the incorporation of an obstacle constraint. We conclude the article with numerical experiments in Section 7, that serve to demonstrate the practical performance of our method.

\section{Abstract framework}

\subsection{An abstract discretization result} 
In an abstract setting, we are concerned with the numerical minimization of the energy functional
\[
I[y] = \frac 1 2 a(y,y) - \a \int_\o \Delta y \cdot(\p_1 y \times \p_2 y) \dv{x}
\]
over a subset of $H^2_\GD(\o) \subset H^2(\o)$ satisfying the boundary conditions $y|_{\GD} = y_\DD$, $\nabla y|_{\GD} = \phi_\DD$ and subject to a (possibly nonlinear) constraint $G[y] = 0 \text{ in } \o$, $G \in C(H^1(\o),L^1(\o))$, which we include in the energy functional by setting $I[y] = \infty$ for all $y \in H^2_\GD(\o)$ for which $G[y]$ is not constantly zero in $\o$. Here, $a$ is a bilinear form that is bounded with constant 1 and coercive on $H^2_\GD$. In this abstract setting, we consider for $h>0$ the discretized constraint $G_h[y_h] = 0$ together with a discretization of the energy functional 
\[
I_h[y_h] = \frac 1 2 a_h(y_h,y_h) - \a \int_\o Q_h[\Delta_h y_h \cdot(\p_1 y_h \times \p_2 y_h)] \dv{x},
\]
defined on a finite dimensional subspace $V_h \subset H^1(\o) \cap C(\oo)$ with a well-defined interpolation operator $\cI_{V_h}: \widetilde{V} \to V_h$ for some dense subset $\widetilde{V} \subset \{y \in H^2(\o) : G[y] = 0 \text{ in } \o \}$, such that $\cI_{V_h}[v]$ converges to $v$ in $H^1(\o)$ as $h$ tends to zero for all $v \in \widetilde{V}$.
Analogously to the continuous case, $a_h$ is a bilinear form that is coercive on $V_h$ and uniformly bounded for all $h>0$. Furthermore, $Q_h$ is a (quasi-)interpolation operator on $L^1(\o)$ and $\Delta_h$ is an operator that is well defined on $V_h$ and approximates the Laplacian if applied to functions in $H^2(\o)$. Under the assumption that the boundary conditions can be satisfied exactly in the discrete space $V_h$, we obtain the following result. 
\begin{theorem}\label{thm:abstr}
  Assume that the discretization $G_h$ of the constraint map satisfies the compatibility conditions $G_h[\cI_{V_h}[y]] = 0$ in $\o$ for all $y \in \widetilde{V}$ and $\|G[y_h]\|_{L^1(\o)} \to 0$ for every sequence $(y_h)$ with $y_h \in V_h$ and $G_h[y_h] = 0$. Let the interpolation operator $Q_h$ be chosen in such a way, that
  \[
  \int_\o Q_h[\Delta_h y_h \cdot(\p_1 y_h \times \p_2 y_h)] \dv{x} - \int_\o \Delta y \cdot(\p_1 y \times \p_2 y) \dv{x} \to 0
  \]
  whenever the sequence $(a_h(y_h,y_h))_h$ is bounded and $y_h \to_{H^1} y$ for $y \in H^2(\o)$.
  If the discrete bilinear form $a_h$ is chosen such that $a_h(\cdot,\cdot)$ converges to the square of the $H^2$-seminorm in the sense of $\G$-convergence and such that for every $y \in \widetilde{V}$ the sequence $(y_h) \subset H^1$ with $y_h = \cI_{V_h}[y] \in V_h$ is a recovery sequence, i.\,e. $\limsup_{h \to 0} a_h(y_h,y_h) \le |y|_{H^2(\o)}$, then the discrete energy functionals $I_h$ converge to the energy $I$, also in the sense of $\G$-convergence with respect to strong convergence in $H^1(\o)$.
  As a consequence, every accumulation point $y$ of a sequence of (almost) minimizers $(y_h)_{h>0}$ of $I_h$ is a minimizer of $I$, belongs to $H^2(\o)$ and satisfies the boundary conditions as well as $G[y]=0$ almost everywhere in $\o$.
\end{theorem}
\begin{proof}
  The proof of the $\G$-convergence result consists of two parts: (i) the \emph{asymptotic lower bound} property, i.\,e. $I[y] \le \liminf_{h \to 0} I_h[y_h]$ holds for every sequence $(y_h) \subset H^1(\o)$ with $y_h \to_{H^1} y$, and (ii) the \emph{existence of recovery sequences}, i.\,e. for every $y \in H_\GD^2(\o)$ with $G[y]=0$, there exists a sequence $(y_h)$, such that $y_h \in V_h$ and $G_h[y_h] = 0$ and $I[y] \ge \limsup_{h \to 0} I_h[y_h]$.
  
  (i) Let $(y_h)_h$ be a sequence with $y_h \in V_h$, such that $y_h \to_{H^1} y$. We may assume that $I_h[y_h] \le C$ uniformly in $h$ (perhaps for a subsequence), since otherwise we have $\liminf_{h \to 0} I_h[y_h] = \infty$ and there is nothing to be shown. We establish the $\liminf$-inequality by showing that
  \[
  a(y,y) \le \liminf_{h \to 0} a_h(y_h,y_h)
  \]
  and
  \[
  \int_\o Q_h[\Delta_h y_h \cdot(\p_1 y_h \times \p_2 y_h)]\dv{x} \to \int_\o \Delta y \cdot(\p_1 y \times \p_2 y) \dv{x}
  \]
  as $h \to 0$.
  From the assumed $\G$-convergence of the sequence $(a_h(\cdot,\cdot))_h$ we have that
  \[
  \|D^2 y\|_{L^2(\o)}^2 \le \liminf_{h \to 0} a_h(y_h,y_h),
  \]
  and hence $a(y,y) \le \liminf_{h \to 0} a_h(y_h,y_h)$ by the boundedness of the bilinear form.
  The convergence of the second term in the energy functionals follows from our assumptions, since we have $a_h(y_h,y_h) \le c_b \|y_h\|^2_{H^1(\o)}$ from the uniform bound on $a_h$ and since the convergent sequence $(y_h)_h$ is bounded in $H^1(\o)$. Since we have $G_h[y_h] = 0$, the compatibility conditions for the discretized contraint imply that $G[y] = 0$. Thus, we deduce that $I[y] \le \liminf_{h \to 0} I_h[y_h]$.
  
  (ii) Let $y \in H^2(\o)$ with $G[y] = 0$ in $\o$. We may assume that $y \in \widetilde{V}$, since $\widetilde{V}$ is dense in the set $\{y \in H^2(\o) : G[y] = 0 \text{ in } \o \}$. We have, by assumption, that the choice of the sequence $(y_h)_h$ with $y_h = \cI_{V_h}[y] \in V_h$ satisfies $G_h[y_h] = 0$ as well as 
  \[
  a(y,y) = \| D^2 y \|^2_{L^2(\o)} \ge \limsup_{h \to 0} a_h(y_h,y_h).
  \]
  The convergence of the second terms in the energy functionals is a consequence of the convergence $y_h \to_{H^1} y$ and our assumptions on the interpolation operator $Q_h$ as in (i). Hence, it follows that $I[y] \ge \limsup_{h \to 0} I_h[y_h]$.
\end{proof}

\begin{remark}
  In the setting of Theorem~\ref{thm:abstr} the isometry constraint can be perturbed and imposed as an inequality. This avoids assuming density of smooth isometries which is especially restrictive when compatibility of given boundary conditions is required. A simple regularization or quasi-interpolation can thus be used instead of the nodal interpolation operator for the definition of a recovery sequence. The relaxation of the isometry constraint is compatible with the discretization in Section~\ref{sec:discr} that forms the basis of our numerical scheme. 
\end{remark}

\section{Preliminaries}

\subsection{Discrete Kirchhoff triangles} 
We employ a spatial discretization based on DKT elements which describe the deformed surface of the plate via its nodal displacement and tangent vectors. In the following we consider a regular triangulation~$\Th$ of the domain~$\o$ into triangles with the index~$h$ denoting the maximum of the diameters $h_T$ for the triangles $T \in \Th$. We let~$\Nh$ and~$\Eh$ denote the set of vertices and edges of elements, respectively, and assume that the Dirichlet boundary~$\GD$ is matched exactly by a subset of~$\Eh$. 
For an integer $k\ge0$ we let~$P_k(T)$ be the set of polynomials of degree less than or equal to~$k$ on~$T \in \Th$ and define the space 
\[
\Pred(T) = \Bigl\{ p \in P_3(T) : p(x_T) = \tfrac{1}{6}\!\!\! \sum_{z \in \Nh \cap T}\bigl(2p(z) - \nabla p(z) \cdot [z - x_T]\bigr) \Bigr\},
\]
where one degree of freedom has been eliminated by prescribing the function value at the center of gravity $x_T = \frac{1}{3}\sum_{z \in \Nh \cap T}z$ of $T$. We then define the finite element spaces
\[
\Sdkt(\Th) = \bigl\{w_h \in C(\oo) : w_h|_T \in \Pred(T) \text{ for all } T \in \Th \text{ and } \nabla w_h \text{ is continuous in } \Nh \bigr\},
\]
and
\[
\cS^2(\Th) = \bigl\{ \theta_h \in C(\oo) : \theta_h|_T \in P_2(T) \text{ for all } T \in \Th \bigr\}.
\]
Note that the degrees of freedom for functions in $\Sdkt(\Th)$ are given by the function values and derivatives at the vertices $\Nh$ of the triangulation, cf.~\cite{bartels2015},~\cite{braess_2007} for details. Thus, the interpolant $\Idkt \colon H^3(\o) \to \Sdkt(\Th)$ given by $\Idkt[y](z) = y(z)$ as well as $\nabla\Idkt[y](z) = \nabla y(z)$ for all $z \in \Nh$ is well defined.
Since $\Idkt$ is exact on the space $P_2(T)$ of polynomials of degree at most two on an element $T \in \Th$, the Bramble-Hilbert lemma yields the interpolation estimate (cf.~\cite[Theorem 4.4.4]{brennerscott2008})
\begin{equation}\label{eq:interpol-dkt}
\| w - \Idkt w \|_{L^p(T)} + h_T \| \nabla w - \nabla \Idkt w \|_{L^p(T)} + h_T^2 \| D^2 w - D^2 \Idkt w \|_{L^p(T)} \le c h_T^3 \| D^3 w\|_{L^p(T)}
\end{equation}
for all $w \in W^{3,p}(T)$ and $1 \le p < \infty$ with a constant $c>0$ not depending on $h_T$.

\subsection{Discrete gradient operator} For an edge $E \in \Eh$ let~$z_E^1, z_E^2 \in \Nh$ denote its two endpoints, $z_E = \frac{1}{2}(z_E^1 + z_E^2)$ its midpoint, $t_E$ a normalized tangent vector and $n_E$ a unit normal to~$E$. 
The \emph{discrete gradient operator} $\nabla_h \colon \Sdkt(\Th) \to \cS^2(\Th)^2$ is then defined as the operator that maps a function $y_h \in \Sdkt(\Th)$ to the uniquely defined function $\theta_h = \nabla_h y_h \in \cS^2(\Th)^2$ satisfying
\begin{align*}
\theta_h(z) &= \nabla y_h(z), \\
\theta_h(z_E) \cdot t_E &= \nabla y_h(z_E) \cdot t_E,\\
\theta_h(z_E) \cdot n_E &= \frac{1}{2}[\nabla y_h(z_E^1) + \nabla y_h(z_E^2)]\cdot n_E,
\end{align*}
for all $z \in \Nh$ and $E \in \Eh$. The mapping can naturally be extended to functions $y \in H^3(\o) \subset C^1(\o)$ by applying the discrete gradient to the interpolant $\Idkt y$. As an immediate consequence of the definition we have the equality 
\[
\nabla_h y_h (z_E) = \frac{1}{2} ([\nabla y_h(z_E^1) + \nabla y_h(z_E^2)]\cdot n_E)n_E + (\nabla y_h(z_E) \cdot t_E)t_E 
\]
for every edge $E \in \Eh$. The result stated below essentially follows the arguments presented in~\cite{bartels2015, bartels2017}.

\begin{proposition}[properties of the discrete gradient operator] \label{prop:dgrad}
  Let $(\Th)_{h>0}$ be a sequence of regular triangulations. There exist constants $c_1, c_2, c_3, c_4 > 0$, independent of $h$ and $h_T$, such that the following estimates involving the discrete gradient operator hold:\\
  (i) For all $w_h \in \Sdkt(\Th)$ we have that
  \[
  c_1^{-1} \| \nabla w_h \|_{L^2(\o)} \le \| \nabla_h w_h \|_{L^2(\o)} \le \| c_1\nabla w_h \|_{L^2(\o)}.
  \]
  (ii) For all $w_h \in \Sdkt(\Th)$ and $T \in \Th$ we have that
  \[
  c_2^{-1} \| D^2 w_h \|_{L^2(T)} \le \| \nabla \nabla_h w_h \|_{L^2(T)} \le \| c_2 D^2 w_h \|_{L^2(T)}.
  \]
  (iii) For all $w \in H^3(\o)$ and $T \in T_h$ we have that
  \[
  \| \nabla w - \nabla_h w \|_{L^2(T)} + h_T\| D^2 w - \nabla \nabla_h w \|_{L^2(T)} \le c_3 h_T^2\| D^3 w \|_{L^2(T)}.
  \]
  (iv) For all $w_h \in \Sdkt(\Th)$ and $T \in \Th$ we have that
  \[
  \|\nabla_h w_h - \nabla w_h\|_{L^2(T)} \le c_4 h_T \| \nabla \nabla_h w_h \|_{L^2(T)}.
  \]
\end{proposition}
\begin{proof}
  (i) It suffices to show that $\| \nabla w_h \|_{L^2(\o)} = 0$ if and only if $\| \nabla_h w_h \|_{L^2(\o)} = 0$. Given $w_h \in \Sdkt(\Th)$ with $\| \nabla w_h \|_{L^2(\o)} = 0$ we deduce that $\nabla w_h|_T = 0$ for all $T \in \Th$. 
  The definition of the discrete gradient immediately implies that we have $\nabla_h w_h|_T = 0$ for all $T \in \Th$ and, hence, $\| \nabla_h w_h \|_{L^2(\o)} = 0$. On the other hand, if we are given $w_h \in \Sdkt(\Th)$ such that $\| \nabla_h w_h \|_{L^2(\o)} = 0$, we deduce that $\nabla w_h(z) = 0$ for all $z \in \Nh$ and $\nabla w_h(z_E) = 0$ for all $E \in \Eh$. 
  In particular, the tangential derivatives of $w_h$ vanish at the endpoints and midpoints of $E$. Since $w_h$ is a cubic polynomial on the straight line extending $E$, we deduce that $w_h|_E$ is constant for all $E \in \Eh$. Noting that $w_h|_T \in \Pred(T)$ for every $T \in \Th$, i.\,e. the last remaining degree of freedom for the cubic polynomial $w_h|_T$ is prescribed, we infer that $w_h$ is constant on $\o$ and, hence, $\| \nabla w_h \|_{L^2(\o)} = 0$.\\
  (ii) Analogously to (i), it suffices to show that $\| D^2 w_h \|_{L^2(T)} = 0$ if and only if $\| \nabla \nabla_h w_h \|_{L^2(T)} = 0$. Given $w_h \in \Sdkt(\Th)$ with $\| D^2 w_h \|_{L^2(T)} = 0$ it follows that $\nabla w_h$ is constant and the definition of the discrete gradient operator immediately implies that $\nabla_h w_h$ is also constant yielding $\| \nabla \nabla_h w_h \|_{L^2(T)} = 0$.
  Conversely, if $\| \nabla \nabla_h w_h \|_{L^2(T)} = 0$, we infer that $\nabla_h w_h|_T$ is constant. Hence, $\nabla w_h(z)$ equals the same constant for all $z \in T \cup \Nh \cup \{ z_E : E \in \Eh \}$. Since $\nabla w_h|_T$ is a quadratic polynomial for every $T \in \Th$, it follows that $\nabla w_h$ is constant in $T$ and, therefore, $\| D^2 w_h \|_{L^2(T)} = 0$.\\
  (iii) If $\nabla w \in P_1(T)^2$ the interpolation obtained with the discrete gradient is exact, i.\,e. in this case we have that $\nabla w|_T = \nabla_h w|_T$. Thus, the Bramble-Hilbert lemma yields for $\nabla w \in H^2(\o)$ the asserted interpolation estimate. \\
  (iv) The estimate follows from (iii) and the inverse estimate $\| D^3 w_h \|_{L^2(T)} \le c h_T^{-1} \| D^2 w_h \|_{L^2(T)}$ for $w_h|_T \in \Pred(T)$.
\end{proof}

\begin{remark}
Note that as a consequence of property (iv) of the above proposition, we have that the mapping $y_h \mapsto \| \nabla \nabla_h y_h \|_{L^2(\o)}$ defines a seminorm on $\Sdkt(\Th)$ and a norm on every subspace of $\Sdkt(\Th)$ with prescribed clamped boundary conditions on $\GD$.
\end{remark}

\section{Discretization} \label{sec:discr}

\subsection{Discrete energy}
For our iterative minimization scheme we consider the discrete energy given for $y_h \in \cA_h$  by
\begin{equation}\label{eq:energy-disc}
  \tEh[y_h] = \frac 1 2 \int_\o |\nabla \nabla_h y_h|^2 \dv{x} - \a \int_\o \hIh^1 \{ \D_h y_h \cdot[\p_1 y_h \times \p_2 y_h] \}\dv{x} - \int_\o \hIh^1[f y_h] \dv{x},
\end{equation}
with the elementwise nodal interpolation operator $\hIh^1$ into piecewise linear $P1$-functions and where the discrete Laplacian $\D_h$ is defined via $\D_h = \diver \nabla_h$. The \emph{discrete admissible set} $\cA_h$ is defined analogously to the continuous case as 
\begin{align*}
  \cA_h = \{ y_h \in \Sdkt(\Th)^3 : \>&y_h(z) = y_\DD(z),\> \nabla y_h(z) = \phi_\DD(z) \text{ for all } z\in\Nh\cap\GD,\\
  & [\nabla y_h(z)]^\top \nabla y_h(z) = I_2 \text{ for all } z \in \Nh \}
\end{align*}
and its \emph{tangent space at a point} $y_h \in \cA_h$ as
\begin{align*}
  \cF_h[y_h] = \{ w_h \in \Sdkt(\Th)^3 : \>& w_h(z) = 0,\> \nabla w_h(z) = 0 \text{ for all } z\in\Nh\cap\GD,\\
  & [\nabla w_h(z)]^\top \nabla y_h(z) + [\nabla y_h(z)]^\top \nabla w_h(z) = 0 \text{ for all } z \in \Nh \}.
\end{align*}
If $y_h \notin \cA_h$, we set $\tEh[y_h] = \infty$. For ease of presentation, the body force $f$ is assumed to vanish in the following, its inclusion is straightforward.

\begin{proposition}[equicoercivity]\label{prop:coerc}
  Let the boundary data satisfy $y_\DD = \widetilde{y}_\DD|_\GD$ and $\phi_\DD = \nabla\widetilde{y}_\DD|_\GD$ for some $\widetilde{y}_\DD \in H^3(\o)^3$, and let $(y_h)_{h>0} \subset H^1(\o)^3$ be a sequence discrete of displacements.
  If the sequence $\tE_h[y_h]$ is uniformly bounded,
  then there exists a constant $C > 0$, independent of $h$, such that
  \[
  \| \nabla \nabla_h y_h \|_{L^2(\o)} \le C.
  \]
\end{proposition}
\begin{proof}
  We first note that $y_h \in \cA_h$, since otherwise we have $\tE_h[y_h] = \infty$. It follows that $|\p_i y_h(z)| = 1$ for $i = 1,2$ and all $z \in \Nh$, which leads to the inequality
  \[
  \tE_h[y_h] \ge \frac{1}{2} \| \nabla \nabla_h y_h \|^2_{L^2(\o)} - \a \int_\o \hIh^1\{ |\D_h y_h| \} \dv{x}.
  \]
  For all $T \in \Th$, we have
  \[
  \int_T \hIh^1\{ |\D_h y_h| \} \dv{x} \le c |T| \| \D_h y_h \|_{L^\infty(T)} \le c' |T| \| \nabla \nabla_h y_h \|_{L^\infty(T)} \le c' c_I \| \nabla \nabla_h y_h \|_{L^2(T)},
  \]
  where the last estimate is a consequence of an inverse estimate for $\nabla_h y_h \in P_2(T)^3$ with a constant $c_I> 0$ not depending on $h$. The validity of the asserted bound follows as a consequence.
\end{proof}

\begin{theorem}[$\G$-convergence]
Let the boundary data satisfy $y_\DD = \widetilde{y}_\DD|_\GD$ and $\phi_\DD = \nabla\widetilde{y}_\DD|_\GD$ for some $\widetilde{y}_\DD \in H^3(\o)^3$. Then we have the following properties:\\
\emph{(i) Common asymptotic lower bound:} For all sequences $(y_h) \subset H^1(\o)^3$ with $y_h \to_{H^1} y$ for some  $y \in H^1(\o)^3$ as $h \to 0$, we have that
\[
\liminf_{h \to 0} \tE_h[y_h] \ge \tE[y].
\] 
Furthermore, if $\liminf_{h \to 0} \tE_h[y_h] < \infty$, the limit $y$ satisfies the boundary conditions, is an element of $H^2(\o)^3$ and we have that  $[\nabla y]^\top\nabla y = I_2$ almost everywhere in $\o$, i.\,e. we have $y \in \cA$.\\
\emph{(ii) Existence of a recovery sequence:} For every admissible $y \in \cA$, there exists a sequence $(y_h)_{h>0}$ with $y_h \in \cA_h$, which converges to $y$ in $H^1(\o)^3$ as $h \to 0$ and for which we have that
\[
\limsup_{h \to 0} \tE_h[y_h] \le \tE[y].
\]
\end{theorem}
\begin{proof} We prove the two asserted properties separately.
  
(i) We may assume that $\liminf_{h \to 0} \tE_h[y_h] < \infty$, otherwise there is nothing to show. Hence, by passing to a subsequence, which we will also label with the index $h$, we may assume that $\tE_h[y_h] \le C$ holds for all $h>0$ with some fixed constant $C>0$. Thus, the coercivity of the discrete energies (Proposition~\ref{prop:coerc}) implies, that the sequence $(\nabla_h y_h)_h$ is uniformly bounded in $H^1(\o)^{3 \times 2}$. 
We therefore deduce the existence of some $\phi \in H^1(\o)^{3 \times 2}$, such that after another passage to a subsequence, once again not relabeled, we have $\nabla_h y_h \wto_{H^1} \phi$ as well as $\nabla_h y_h \to_{L^2} \phi$ as $h$ tends to $0$.
With Proposition~\ref{prop:dgrad}(iv) we see that
\begin{align*}
\| \nabla y - \nabla_h y\|_{L^2(\o)} &\le \|\nabla y - \nabla y_h\|_{L^2(\o)} + \| \nabla y_h - \nabla_h y_h \|_{L^2(\o)}\\
&\le \|\nabla y - \nabla y_h\|_{L^2(\o)} + c_4 h \| \nabla \nabla_h y_h \|_{L^2(\o)}.
\end{align*}
Since $y_h \to_{H^1} y$ and hence $\nabla y_h \to \nabla y$ holds by assumption, passing to the limit $h \to 0$ in the above yields that $\nabla y = \phi$ and thus $y \in H^2(\o)^3$. Since we assumed the boundary data to be sufficiently regular, we have that $y|_\GD = y_\DD$ and $\nabla y|_\GD = \phi_\DD$, i.\,e. the limit $y$ satisfies the boundary conditions. Since $y_h \in \cA_h$, we have
\[
\| [\nabla_h y_h]^\top \nabla_h y_h - I_2 \|_{L^1(\o)} \le ch\| \nabla ([\nabla_h y_h]^\top \nabla_h y_h) \|_{L^1(\o)}
\]
by a discrete interpolation estimate. The uniform bound of $\nabla_h y_h$ in $H^1(\o)^{3 \times 2}$ from Proposition~\ref{prop:coerc} implies that the term on the right hand side converges to zero as $h \to 0$. It follows that $[\nabla_h y_h]^\top \nabla_h y_h \to I_2$ alost everywhere in $\o$.
Since we also have that $[\nabla_h y_h]^\top \nabla_h y_h \to [\nabla y]^\top \nabla y$ almost everywhere in $\o$, this yields $[\nabla y]^\top \nabla y = I_2$ almost everywhere in $\o$ and we thus have established that $y \in \cA$. To see that $\tE[y]$ is a common asymptotic lower bound, we first note that by the weak lower semi-continuity of the $H^1$-seminorm we have from the convergence $\nabla_h y_h \rightharpoonup_{H^1} \nabla y$ that
\[
\int_\o |D^2 y|^2 \dv{x} = \int_\o |\nabla \nabla y|^2 \dv{x} \le \liminf_{h \to 0} \int_\o |\nabla \nabla_h y_h|^2 \dv{x}.
\]
To deduce the $\liminf$-inequality, we show that the term
\[
R = \int_\o \hIh^1 \{ \D_h y_h \cdot[\p_1 y_h \times \p_2 y_h] \}\dv{x} -  \int_\o \D y \cdot[\p_1 y \times \p_2 y] \dv{x}
\]
converges to zero as $h \to 0$. We split the term via $R = R_1 + R_2 + R_3$ with 
\begin{align*}
  R_1 &= \int_\o \hIh^1 \{ \D_h y_h \cdot[\p_1 y_h \times \p_2 y_h] \} - \D_h y_h \cdot[\p_1 y_h \times \p_2 y_h]\dv{x},\\ 
  R_2 &=  \int_\o [\D_h y_h - \D y] \cdot[\p_1 y \times \p_2 y]\dv{x},\\
  R_3 &= \int_\o \D_h y_h \cdot[\p_1 y_h \times \p_2 y_h - \p_1 y \times \p_2 y] \dv{x}.
\end{align*}
 For the term $R_1$ we have that
\begin{align*}
  &\int_T \hIh^1 \{ \D_h y_h \cdot[\p_1 y_h \times \p_2 y_h] \} - \D_h y_h \cdot[\p_1 y_h \times \p_2 y_h]\dv{x} \\
  &\le \, c_\cI h^2 \| D^2(\Delta_hy_h \cdot(\p_1 y_h \times \p_2 y_h))\|_{L^1(T)}
\end{align*}
 as a consequence of nodal interpolation estimates on the elements $T \in \Th$, since all the involved functions are polynomials. 
 Expanding the second derivative with a chain rule and using that $D^2\Delta_h y_h = 0$ on every $T \in \Th$, we see that
 \begin{align*}
 &\| D^2(\Delta_hy_h \cdot(\p_1 y_h \times \p_2 y_h))\|_{L^1(T)} \\ &\, \le 2 \| \nabla \Delta_h y_h \cdot (\nabla ( \p_1 y_h \times \p_2 y_h )) \|_{L^1(T)} + \| \Delta_h y_h D^2( \p_1 y_h \times \p_2 y_h ) \|_{L^1(T)}.
\end{align*}
Using Hölder's inquality and inverse estimates for the polynomials, the right hand side of this inequality can be bound in terms of $c_I h^{-1} \| \Delta_h y_h \|_{L^2(T)} \| \p_1 y_h \times \p_2 y_h \|_{L^2(T)}$. The uniform bound on $\|\nabla \nabla_h y_h\|_{L^2(\o)}$ from Proposition~\ref{prop:coerc} thus implies the convergence of $R_1$ to zero.
 The weak convergence $\nabla_h y_h \wto_{H^1} \nabla y$ implies that $R_2$ tends to zero as $h \to 0$. For the third term, we have with Hölder's inequality and the uniform bound on $\|\nabla \nabla_h y_h\|_{L^2(\o)}$ from Proposition~\ref{prop:coerc} that
\begin{align*}
R_3 &\le \|\D_h y_h\|_{L^2(\o)} \| \p_1 y_h \times \p_2 y_h - \p_1 y \times \p_2 y \|_{L^2(\o)} \\
& \le C \| \p_1 y_h \times \p_2 y_h - \p_1 y \times \p_2 y \|_{L^2(\o)}.
\end{align*}
Hence, the convergence $y_h \to_{H^1} y$ implies the convergence $R_3 \to 0$ as $h$ tends to zero.

(ii) Since isometries in $\cA \subset H^2(\o)^3$ can be approximated with arbitrary precision in the $H^2$-norm by smooth isometries (cf.~\cite{hornung}) and since the energy $\tE$ is continuous on $H^2(\o)^3$, we may without loss of generality assume that $y \in \cA \cap H^3(\o)^3$. We define the recovery sequence $(y_h)_{h>0}$ via $y_h = \Idkt y$. 
By definition the boundary conditions are satisfied and we have $[\nabla y_h(z)]^\top \nabla y_h(z) = I_2$ for all $z \in \Nh$ and thus, $y_h \in \cA_h$ for all $h > 0$. The convergence $y_h \to_{H^1} y$ as $h \to 0$ is a consequence of the interpolation estimate~\eqref{eq:interpol-dkt}, which yields that
\[
\| y_h - y \|_{H^1(\o)} \le c h^2 \| y \|_{H^3(\o)}
\]
with a constant $c>0$ independent from $h$. We will now go on to show that $\tE_h[y_h]$ converges to~$\tE[y]$ as $h \to 0$.
Since $y_h$ is the interpolant of $y$ in $\Sdkt(\Th)^3$, we have that $\nabla \nabla_h y_h = \nabla \nabla_h y$. Thus, Proposition~\ref{prop:dgrad}(iii) yields 
\[
\| \nabla \nabla_h y_h - D^2 y \|_{L^2(\o)} \le c_3 h \| D^3y \|_{L^2(\o)}
\]
for all $T \in \Th$, where the right hand side tends to zero as $h \to 0$, which proves the convergence for the first terms of~\eqref{eq:energy-re} and~\eqref{eq:energy-disc}. To obtain convergence of the second terms, we note that 
\[
\hIh^1\{ \D_h y_h \cdot [\p_1 y_h \times \p_2 y_h] \} = \hIh^1\{ \D_h y \cdot [\p_1 y \times \p_2 y]
\]
and split the residual
\[
R = \int_\o \hIh^1\{ \D_h y_h \cdot [\p_1 y_h \times \p_2 y_h] \} \dv{x} - \int_\o \D y \cdot [\p_1 y \times \p_2 y] \dv{x}
\]
via $R = R_1 + R_2$ with
\begin{align*}
  R_1 &=  \int_\o \hIh^1\{ \D_h y_h \cdot [\p_1 y_h \times \p_2 y_h] \} - \D_h y_h \cdot [\p_1 y_h \times \p_2 y_h]\dv{x},\\
  R_2 &=\int_\o \D_h y \cdot [\p_1 y \times \p_2 y] - \int_\o \D y \cdot [\p_1 y \times \p_2 y] \dv{x}.
\end{align*}
As in part (i) of the proof, the term $R_1$ converges to zero as $h$ tends to zero as a consequence of nodal interpolation estimates and inverse inequalitues on the elements $T \in \Th$ and the fact that $\|\Delta_h y_h\|_L^2(\o)$ is bounded. For $R_2$, we have with Hölder's inqeuality that 
\[
R_2 \le \| \D_h y - \D y \|_{L^2(\o)} \| \p_1 y \times \p_2 y \|_{L^2(\o)},
\]
where $\| \p_1 y \times \p_2 y \|_{L^2(\o)} = |\o|^{(1/2)}$ since $|\p_1 y \times \p_2 y| = 1$ due to the fact that $y$ is an isometry. Furthermore, we have that 
\[
\| \D_h y - \D y |_{L^2(\o)} \le 2 \| \nabla \nabla_h y - D^2 y \|_{L^2(\o)} \le c_3 h \| D^3y \|_{L^2(\o)},
\]
where the second inequlity follows from Proposition~\ref{prop:dgrad}(iii). Thus, $R_2$ also tends to zero as $h \to 0$, concluding the proof.
\end{proof}

\section{Minimization of the discrete energies}\label{sec:gradflow}
\subsection{Discrete gradient flow}
In this section we propose a semi-implicit discrete gradient flow scheme, which employs a linearization of the isometry constraint in every (pseudo-)timestep, for the minimization of the discrete energies~\eqref{eq:energy-disc}. This numerical method is outlined in~\cite{bartels2020} without an in-depth investigation. We prove that the resulting algorithm is energy decreasing and that the discrete solutions obtained with the scheme satisfy the isometry constraint up to a small error, which only depends on the step size and specifically is independent of the number of iterations of the algorithm. The practical properties of the proposed method are illustrated with numerical experiments in Section~\ref{sec:numexp}.

We denote with $(\cdot,\cdot)_* = (\nabla \nabla_h \cdot, \nabla \nabla_h \cdot)$ the scalar product used to define the gradient flow and its induced norm  with $\| \cdot \|_*$.

\begin{algorithm}[linearized isometry flow]\label{alg:isoflow}
 Choose a termination criterion $\estop>0$, a step size $\tau > 0$, an initial value $y_h^0 \in \cA_h$, and set $k = 1$.\\
(1) Compute $d_t y_h^k \in \cF_h[y_h^{k-1}]$ such that 
\begin{equation}\label{eq:update}
\begin{aligned}
(d_ty_h^k, w_h)_*  = -&(\nabla \nabla_h y_h^{k-1}, \nabla \nabla_h w_h) - \tau(\nabla \nabla_h d_t y_h^k, \nabla \nabla_h w_h) \\ 
 &+ \a \int_\o \hIh^1\{ \D_h w_h \cdot [\p_1 y_h^{k-1} \times \p_2 y_h^{k-1}] \} \dv{x} \\
 &+ \a \int_\o \hIh^1\{ \D_h y_h^{k-1} \cdot [ \p_1 w_h \times \p_2 y_h^{k-1}] \} \dv{x} \\
 &+ \a \int_\o \hIh^1\{ \D_h y_h^{k-1} \cdot [ \p_1 y_h^{k-1} \times \p_2 w_h ] \} \dv{x}
\end{aligned}
\end{equation}
for all $w_h \in \cF_h[y_h^{k-1}]$.\\
(2) Set $y_h^k = y_h^{k-1} + \tau d_t y_h^k$. Stop the iteration if $\|d_t y_h^k\|_* \le \estop$. Otherwise increase~$k$ via $k := k+1$ and continue with (1).
\end{algorithm}

The practical properties of the iterates obtained with the above algorithm are put together in the following theorem, which also imply the termination of the algorithm after a finite number of steps. To simplify the notation, we write
\[
\| v \|_{L_h^p} = \Bigl(\sum_{T \in \Th} \frac{|T|}{3} \sum_{z \in \Nh \cap T} |\hIh^1 v(z)|^p\Bigr)^{(1/p)}
\]
for the discrete $L^p$ norm of a function $v \colon \o \to \R^\ell$ obtained by element-wise linear interpolation of the integrand $v$. In the case $p=\infty$, the discrete norm is defined via $\| v \|_{L_h^\infty} = \max_{T \in \Th} \| v \|_{L^\infty(T)}$.

\begin{theorem}(iteration)\label{thm:iteration}
The iterates $(y_h^L)_{L=0,1,\dotsc}$ of Algorithm~\ref{alg:isoflow} are well defined and satisfy
\begin{equation}\label{eq:energ-bound}
\tE_h[y_h^L] + (1- C \tau |\log h_\mathrm{min}|) \tau \sum_{k=1}^L\|d_t y_h^k  \|_*^2 \le \tE_h[y_h^0]
\end{equation}
with $h_\mathrm{min} = \min_{T \in \Th}(h_T)$ and a constant $C > 0$ independent of $L$. Furthermore, if the step size $\tau$ is chosen small enough, such that $\tau \le (2 C | \log h_\mathrm{min}|)^{-1}$, then we have
\begin{equation}\label{eq:constviol-bound}
\| [\nabla y_h^L]^\top \nabla y_h^L  -I_2 \|_{L_h^\infty} \le \widetilde{C} \tau |\log h_\mathrm{min} | \tE_h[y_h^0]
\end{equation}
for a constant $\widetilde{C} > 0$ independent of $L$.
\end{theorem}
\begin{proof}
For the proof, we assume that the estimates have been established for all $k \le L-1$ and employ an inductive argument. Choosing $d_t y_h^k = \frac{1}{\tau}(y_h^k - y_h^{k-1})$ as a test function in~\eqref{eq:update}, we see that for $k \le L$ we have
\begin{align*}
  & \| \nabla \nabla_h d_t y_h^k \|_{L^2(\o)}^2  + \frac{1}{2\tau}\bigl( \|\nabla \nabla_h y_h^k \|_{L^2(\o)}^2 - \| \nabla \nabla_h y_h^{k-1}\|_{L^2(\o)}^2 \bigr) \\
  & \le \a \int_\o \hIh^1\{ \D_h d_t y_h^k \cdot [\p_1 y_h^{k-1} \times \p_2 y_h^{k-1}] \} \dv{x} \\
  &\quad + \a \int_\o \hIh^1\{ \D_h y_h^{k-1} \cdot [ \p_1 d_t y_h^k \times \p_2 y_h^{k-1} + \p_1 y_h^{k-1} \times \p_2 d_t y_h^k ] \} \dv{x} \\ 
  & \le \a \| [ \D_h d_t y_h^k ] \|_{L_h^2} \| [\p_1 y_h^{k-1}] \|_{L_h^2} \| [\p_2 y_h^{k-1}] \|_{L_h^\infty} \\
  &\quad + \a \| [ \D_h y_h^{k-1} ] \|_{L_h^2}( \| \p_1 y_h^{k-1} \|_{L_h^\infty} \| \p_2 d_t y_h^k \| _{L_h^2} + \| \p_1 d_t y_h^k \|_{L_h^2} \| \p_2 y_h^{k-1} \| _{L_h^\infty}),
\end{align*}
where the second estimate follows from the Cauchy-Schwarz inequality. Multiplying the above estimate by $\tau$ and using Young's inequality to absorb the right hand side terms invovling $d_t y_h^k$ on the left hand side yields the intermediate estimate
\begin{equation}\label{eq:intermed}
\frac{\tau}{2} \| \nabla \nabla_h d_t y_h^k \|_{L^2(\o)}^2 + \frac{1}{2} \|\nabla \nabla_h y_h^k \|_{L^2(\o)}^2 \le \frac{1}{2 } \| \nabla \nabla_h y_h^{k-1}\|_{L^2(\o)}^2 + \tau c' \le c''
\end{equation}
for some constant $c' > 0$ which is independent of $L$, since our induction hypothesis implies such a bound for all terms involving $y_h^{k-1}$ as consequence of the coercivity of the discrete energy (cf.~Proposition~\ref{prop:coerc}).
For the new iterate $y_h^k = y_h^{k-1} + \tau d_t y_h^k$, we have that
\[
[\nabla y_h^k(z)]^\top \nabla y_h^k(z) = [\nabla y_h^{k-1}(z)]^\top \nabla y_h^{k-1}(z) + \tau^2 [\nabla d_t y_h^k(z)]^\top \nabla d_t y_h^k(z)
\]
at every node $z \in \Nh$, because the term $[\nabla d_t y_h^k(z)]^\top \nabla y_h^{k-1}(z) + [\nabla y_h^{k-1}(z)]^\top \nabla d_t y_h^k(z)$ vanishes for the update $d_t y_h^k \in \cF_h[y_h^{k-1}]$. Note that in the above we may replace the gradient $\nabla d_t y_h^k$ with the discrete gradient $\nabla_h d_t y_h^k$, since by definition of the discrete gradient the nodal values of both expressions coincide. This leads to the estimate
\[
\| [\nabla y_h^k]^\top \nabla y_h^k - I_2 \|_{L_h^\infty} \le \| [\nabla y_h^{k-1}]^\top \nabla y_h^{k-1} - I_2 \|_{L_h^\infty} + \tau^2 \|\nabla_h d_t y_h^k \|_{L_h^\infty}^2.
\]
The discrete Sobolev inequality $\|\nabla_h d_t y_h^k \|_{L_h^\infty}^2 \le c_\mathrm{inv}( 1+  |\log h_\mathrm{min}|) (\| \nabla_h d_t y_h^k \|^2 + \| \nabla \nabla_h d_t y_h^k\|^2)$, cf.~\cite{brennerscott2008}, together with the Poincaré inequality $\|\nabla_h d_t y_h^k\|_{L^2(\o)}^2 \le c_\mathrm{P}\|\nabla \nabla_h d_t y_h^k\|_{L^2(\o)}^2 $,  the intermediate estimate~\eqref{eq:intermed} and the constraint violation bound for $y_h^{k-1}$ imply the estimate
\begin{equation}\label{eq:err-subopt}
\| [\nabla y_h^k]^\top \nabla y_h^k - I_2 \|_{L_h^\infty} \le c''' \tau |\log h_\mathrm{min}| \le \tilde{c},
\end{equation}
with the constant $c''' = \widetilde{C}\tE_h[y_h^0] + 2 c_\mathrm{inv}c_\mathrm{P}c'' > 0$, which will be improved below. 
In order to deduce the asserted energy bound for $y_h^k$ we employ the discrete product rule $d_t(a^k b^k) = (d_t a^k) b^k + a^{k-1}(d_t b^k)$, i.\,e.
\begin{align*}
&\frac{1}{\tau} \Bigl( \int_\o \hIh^1 \{ \D_h y_h^k \cdot[\p_1 y_h^k \times \p_2 y_h^k] \}\dv{x} - \int_\o \hIh^1 \{ \D_h y_h^{k-1} \cdot[\p_1 y_h^{k-1} \times \p_2 y_h^{k-1}] \}\dv{x} \Bigr) \\ 
&= \int_\o \hIh^1 \{ \D_h d_t y_h^k \cdot[\p_1 y_h^k \times \p_2 y_h^k] \}\dv{x} + \int_\o \hIh^1 \{ \D_h y_h^{k-1} \cdot[\p_1 y_h^{k-1} \times \p_2 d_t y_h^k + \p_1 d_t y_h^k \times \p_2 y_h^k] \}\dv{x}
\end{align*}
to rewrite the right hand side of~\eqref{eq:update} with $w_h = d_t y_h^k$ and obtain
\begin{equation}\label{eq:energ-est}
\begin{aligned}
&\| \nabla \nabla_h d_t y_h^k \|_{L^2(\o)}^2  + \frac{1}{2\tau}\bigl( \|\nabla \nabla_h y_h^k \|_{L^2(\o)}^2 - \| \nabla \nabla_h y_h^{k-1}\|_{L^2(\o)}^2 \bigr) + \frac{\tau}{2} \| \nabla \nabla_h d_t y_h^k \|_{L^2(\o)}^2 \\
&= \frac{\a}{\tau} \bigl( \int_\o \hIh^1 \{ \D_h y_h^k \cdot[\p_1 y_h^k \times \p_2 y_h^k] \}\dv{x} - \int_\o \hIh^1 \{ \D_h y_h^{k-1} \cdot[\p_1 y_h^{k-1} \times \p_2 y_h^{k-1}] \}\dv{x} \bigr) \\
&\quad - \a \int_\o \hIh^1 \{ \D_h d_t y_h^k \cdot[\p_1 y_h^k \times \p_2 y_h^k - \p_1 y_h^{k-1}  \times \p_2 y_h^{k-1}]\}\dv{x} \\
&\quad + \a \int_\o \hIh^1 \{ \D_h y_h^{k-1} \cdot [ (\p_1 y_h^{k-1} - \p_1 y_h^k) \times \p_2 d_t y_h^k] \} \dv{x}. \\
\end{aligned}
\end{equation}
After including the mixed term $\pm (\p_1 y_h^{k-1} \times \p_2 y_h^k)$ in the second term on the right hand side of this equation, we see that it is bounded via
\begin{align*}
  &\a \int_\o \hIh^1 \{ \D_h d_t y_h^k \cdot[\p_1 y_h^k \times \p_2 y_h^k - \p_1 y_h^{k-1}  \times \p_2 y_h^{k-1}]\}\dv{x} \\
  &\le \a \| \D_h d_t y_h^k \|_{L_h^2}\bigl( \tau \| \p_1 d_t y_h^k \|_{L_h^2} \| \p_2 y_h^k \|_{L_h^\infty} + \| \p_1 y_h^{k-1}\|_{L_h^\infty} \tau \| \p_2 d_t y_h^k \|_{L_h^2}  \bigr) \\
  &\le c_1 \alpha \tau \| \nabla \nabla_h d_t y_h^k \|_{L^2(\o)}^2,
\end{align*}
where the first estimate results from the Cauchy-Schwarz inequality and the second estimate is a consequence of the bounds $\| \p_j y_h^{k-\ell} \|_{L_h^\infty} \le c$ for $j = 1,2$ and $\ell = 0,1$ together with the bound $\| \nabla d_t y_h^k \|_{L_h^2} \le c \| \nabla \nabla_h d_t y_h^k \|_{L^2(\o)}$.
For the third term on the right hand side of~\eqref{eq:energ-est}, we have that
\begin{align*}
  \a \int_\o \hIh^1 \{ \D_h y_h^{k-1} \cdot [ (\p_1 y_h^{k-1} - \p_1 y_h^k) \times \p_2 d_t y_h^k] \} \dv{x} &\le \alpha \| \D_h y_h^{k-1} \|_{L_h^2} \tau \| \p_1 d_t y_h^k \|_{L_h^2} \| \p_2 d_t y_h^k \|_{L_h^\infty} \\
  & \le c_2 \alpha \tau |\log h_\mathrm{min}| \| \nabla \nabla_h d_t y_h^k \|_{L^2(\o)}^2,
\end{align*}
where the first estimate again results from the Cauchy-Schwarz inequality and the second estimate results from the bounds $\| \p_2 d_t y_h^k \|_{L_h^\infty} \le c |\log h_\mathrm{min}| \|\nabla \nabla_h d_t y_h^k\|_{L^2(\o)}$ and $\| \p_1 d_t y_h^k \|_{L_h^2} \le c \| \nabla \nabla_h d_t y_h^k \|_{L^2(\o)}$ as well as the energy bound for $y_h^{k-1}$. Plugging the above estimates into~\eqref{eq:energ-est} and multiplying with $\tau$ we deduce that 
\[
\tau(1 - C\tau|\log h_\mathrm{min}|) \| \nabla \nabla_h d_t y_h^k \|_{L^2(\o)}^2 + \tE_h[y_h^k] \le \tE_h[y_h^{k-1}]
\]
holds for all $k \le L$ with a constant $C>0$ independent of $k$, which proves the energy estimate~\eqref{eq:energ-bound}. By inductively using this newly established energy bound instead of the intermediate estimate~\eqref{eq:intermed} in the derivation of the constraint violation error~\eqref{eq:err-subopt} we get the optimal estimate~\eqref{eq:constviol-bound} of the constraint violation.
\end{proof}

\section{Obstacle constraints}
\subsection{Convex-concave penalization}
We now consider the minimization problem for the energy~\eqref{eq:energy-re} subject to $y \in \cA$ and subject to the additional obstacle constraint $y_3 \le 1$ in $\o$. We include the obstacle constraint in the energy functional via penalization, i.\,e. we penalize values of the third component of $y$ that exceed the obstacle via the addition of the penalty term
\begin{equation}\label{eq:pen-term}
  P_\varepsilon[y_3] = \frac{1}{2\varepsilon} \int_\o (y_3 -1)_+^2 \dv{x},
\end{equation}
to the energy $\tE[y]$, where $(y_3 -1)_+ = \max\{y_3 -1, 0\}$ and $\varepsilon > 0$ is a small penalization parameter. Note that this approach can easily be generalized to include non-constant obstacles which depend on $x \in \R^2$, i.e. which are given by a two-dimensional function $g \colon \R^2 \to \R$ and result in a constraint of the form $y_3 \le g(y_1,y_2)$. The arguments in the following remain valid in this general case. The penalized energy in the simple constant-obstacle minimization problem is given by
\[
\tE_\mathrm{pen}[y] = \frac 1 2 \int_\o |D^2 y|^2 \dv{x} - \a \int_\o \D y \cdot[\p_1 y \times \p_2 y] \dv{x} - \int_\o f \cdot y \dv{x} + \frac{1}{2\varepsilon} \int_\o(y_3 -1)_+^2 \dv{x}.
\]
We split the integrand in the penalty term into convex and concave parts via
\[
(y_3-1)_+^2 =  (y_3)^2 + P_\mathrm{ccv}(y_3)
\]
with
\[
P_\mathrm{ccv}(s) = 
\begin{cases}
  -2s + 1, &\text{if } s > 1, \\
  -s^2, &\text{if } s \le 1.
\end{cases}
\]
Note that we have $P'_\mathrm{ccv}(s) = p_\mathrm{ccv}(s)$ with the monotonically decreasing continuous function
\[
p_\mathrm{ccv} = 
\begin{cases}
  -2, &\text{if } s>1, \\
  -2s, &\text{if } s \le 1.
\end{cases}
\]
This splitting of the penalty term allows for an implicit treatment of the quadratic, convex part in the discrete gradient flow, while the concave part is treated explicitly in every time step. 
The penalized discrete gradient flow obtained in this way is unconditionally stable, provided that the underlying unpenalized gradient flow is stable. In the following abstract lemma we denote by $(\cdot,\cdot)_*$ a scalar product with induced norm $\|\cdot\|_*$ on a given solution space $W$, e.\,g. $W = H^2(\o)$, with adequately defined linear subspaces $\cF^{k} \subset W$, $k \in \N_0$. We show that the penalized discrete gradient flow for the energy $I[y] = \frac{1}{2} \| y \|_*^2 - (f,y)$ with a given body force $f \in L^2(\o)$ is unconditionally stable.
Note that this result can be extended to the case of the penalized semi-implicit isometry flow for the energy $\tE[y]$. To simplify notation we denote the Gateaux derivative in the direction $w$ of the energy $I$ at a point $y \in W$ with $I'[y; w] = (y,w)_* - (f,w)$. 

\begin{proposition}[unconditional energy decay]\label{prop:endecay}
  The convex-concavely penalized gradient flow defined by $y^k = y^{k-1} + \tau d_t y^k$ with $d_t y^k \in \cF^{k-1}$ such that
  \[
  (d_t y^k,w)_* + I'[y^k,w] + \frac{1}{\varepsilon} (y_3^k,w_3) = - \frac{1}{2\varepsilon} (p_\mathrm{ccv}(y_3^{k-1}),w_3)
  \]
  for all $w \in \cF^{k-1}$ is well defined and satisfies
  \[
  I[y^k] + P_\varepsilon[y_3^k] +  \tau \| d_t{y^k}\|_*^2 \le I[y^{k-1}] + P_\varepsilon[y_3^{k-1}]
  \]
  for all $k \ge 1$.
\end{proposition}
\begin{proof}
  Existence and uniqueness of the updates $d_t y^k$ in every step follow from the Lax-Milgram lemma. Furthermore, expanding the convex part of the penalty term we see that
  \[
  P_\varepsilon[y_3^k] = \frac{1}{2\varepsilon} \| y_3^{k-1} \|_{L_2(\o)}^2 + \frac{\tau}{\varepsilon} (y_3^{k-1}, d_t y_3^k ) + \frac{\tau^2}{2\varepsilon}\| d_t y_3^k \|_{L_2(\o)}^2 - \frac{1}{2\varepsilon} \int_\o P_\mathrm{ccv}[y_3^k] \dv{x}.
  \]
  Since the function $P_\mathrm{ccv}$ is concave, it holds that $
  P'_\mathrm{ccv}(y_3^{k-1}) (y_3^k - y_3^{k-1}) \ge P_\mathrm{ccv}[y_3^k] - P_\mathrm{ccv}[y_3^{k-1}]$, 
  from which we deduce that
  \[
  \frac{1}{2 \varepsilon} \int_\o P_\mathrm{ccv}[y_3^k] \dv{x} \ge \frac{1}{2 \varepsilon} \int_\o P_\mathrm{ccv}[y_3^{k-1}] \dv{x} + \frac{\tau}{2 \varepsilon} \bigl(p_\mathrm{ccv}(y_3^{k-1}),d_t y_3^k\bigr).
  \]
  Thus, it follows that
  \[
  P_\varepsilon[y_3^k] \le P_\varepsilon[y_3^{k-1}] + \frac{\tau}{\varepsilon} \bigl(y_3^{k-1}, d_t y_3^k \bigr) + \frac{\tau^2}{2\varepsilon}\| d_t y_3^k \|_{L_2(\o)}^2 - \frac{\tau}{2 \varepsilon} \bigl(p_\mathrm{ccv}(y_3^{k-1}),d_t y_3^k\bigr).
  \]
  Since we have that  $I'[y^k;d_ty^k] \ge \frac{1}{\tau} \bigl(I[y^k] - I[y^{k-1}]\bigr)$, the choice $w = d_t y_k$ in the discrete gradient flow now proves the assertion.
\end{proof}

The decay property of the penalized energy allows us to estimate the artificial obstacle penetration.

\begin{corollary}\label{cor:penest}
  If the sequence of iterates $(y^k)$ satisfies $\|y_3^k\|_{W^{1,\infty}(\o)} \le C$ for some constant $C>0$ independent of k, then it holds that $\|(y^k_3 - 1)_+\|_{L^\infty(\o)} \le C \varepsilon^{1/4} e_0^{1/4}$, where $e_0 = I[y^0] + P_\varepsilon[y^0_3]$ is the initial energy.
\end{corollary}
\begin{proof}
  We use the Gagliardo–Nirenberg interpolation inequality \cite[Theorem 10.1]{friedman2008} which states that for $\o \subset \R^N$ and $1 \le q \le p \le \infty$ and $r > N$, there exists a constant $c = c(\o)$, such that with $\a = (\frac{1}{q} + \frac{1}{p})/(\frac{1}{q} + \frac{1}{N} - \frac{1}{r})$ the inequality $\|u\|_{L^p(\o)} \le c \| u \|_{L^q(\o)}^{1-\a} \| u \|_{W^{1,r}(\o)}^{\a}$ holds for all functions $u \in W^{1,r}(\o)$.
  The interpolation inequality for the choices $q = 2$ and $r = p = \infty$ applied to the obstacle penetration $u = (y^k_3 - 1)_+$ yields
  \begin{equation*}
    \|(y^k_3 - 1)_+\|_{L^\infty(\o)} \le c \| (y^k_3 - 1)_+ \|_{L^2(\o)}^{1/2} \| (y^k_3 - 1)_+   \|_{W^{1,\infty}(\o)}^{1/2}.
  \end{equation*}
  Since $\| (y^k_3 - 1)_+ \|_{L^2(\o)} = (2\varepsilon P_\varepsilon[y^k_3])^{(1/2)}$, the assertion of the corollary follows from the estimate $P_\varepsilon[y^k_3] \le I[y_0] + P_\varepsilon[y^0_3]$ together with the assumption $\| y^k_3 \|_{W^{1,\infty}(\o)}  \le C$.
\end{proof}

\begin{remark}
  If the penalization is applied to the linearized isometry flow for the energy $\tEh[y_h]$ from Algorithm~\ref{alg:isoflow}, the assumed bound $\|y_3^k\|_{W^{1,\infty}(\o)} \le C$ follows from the bound~\eqref{eq:constviol-bound} on the isometry error.
\end{remark}

\subsection{Discretization}
In the finite element discretization we use numerical integration by mass lumping, i.\,e. we consider the penalized discrete energy
\begin{equation}\label{eq:pen-discenergy}
  \begin{aligned}
    \tEpenh[y_h] = &\frac 1 2 \int_\o |\nabla \nabla_h y_h|^2 \dv{x} - \a \int_\o \hIh^1 \{ \D_h y_h \cdot[\p_1 y_h \times \p_2 y_h] \}\dv{x}\\
    & - \int_\o \Ih^1[f w_h] \dv{x} + \frac{1}{2\varepsilon} \int_\o \Ih^1[(y_{h,3})^2] \dv{x} - \frac{1}{2\varepsilon} \int_\o \Ih^1[P_\mathrm{ccv}(y_3)] \dv{x},
  \end{aligned}
\end{equation}
where $\Ih$ denotes the nodal interpolant in the set of continuous piecewise linear polynomials $\cS^1(\Th) = \{ v_h \in C(\o) : v_h|_T \in P_1(T) \text{ for all } T \in \Th \}$. Note that the statement of Proposition~\ref{prop:endecay} applies also to the discrete setting.
The discrete gradient flow for the accordingly penalized bending energy then takes the following form.

\begin{algorithm}[convex-concavely penalized discrete isometry flow]\label{alg:penisoflow}
 Choose a termination criterion $\estop>0$, a penalization parameter $\varepsilon>0$, a step size $\tau > 0$ and an initial value $y_h^0 \in \cA_h$, and set $k = 1$.\\
(1) Compute $d_t y_h^k \in \cF[y_h^{k-1}]$ such that 
\begin{equation}
\begin{aligned}
&(1+\tau)(\nabla \nabla_h d_t y_h^k, \nabla \nabla_h w_h) + \frac{\tau}{\varepsilon} \int_\o \Ih^1[d_t y_{h,3}^k w_{h,3} ] \dv{x} \\
& = -(\nabla \nabla_h y_h^{k-1}, \nabla \nabla_h w_h) - \frac{1}{\varepsilon} \int_\o \Ih^1[y_{h,3}^{k-1} w_{h,3}] \dv{x} - \frac{1}{2\varepsilon} \int_\o \Ih^1[p_\mathrm{ccv}(y_{h,3}^{k-1}) w_{h,3}] \dv{x} \\
 &\quad + \a \int_\o \hIh^1\{ \D_h w_h \cdot [\p_1 y_h^{k-1} \times \p_2 y_h^{k-1}] \} \dv{x} + \a \int_\o \hIh^1\{ \D_h y_h^{k-1} \cdot [ \p_1 w_h \times \p_2 y_h^{k-1}] \} \dv{x} \\
 &\quad + \a \int_\o \hIh^1\{ \D_h y_h^{k-1} \cdot [ \p_1 y_h^{k-1} \times \p_2 w_h ] \} \dv{x} + \int_\o \Ih^1[f w_h] \dv{x}
\end{aligned}
\end{equation}
for all $w_h \in \cF[y_h^{k-1}]$.\\
(2) Set $y_h^k = y_h^{k-1} + \tau d_t y_h^k$. Stop the iteration if $\|\nabla \nabla_h d_t y_h^k\|_{L^2(\o)} \le \estop$. Otherwise increase~$k$ via $k := k+1$ and continue with (1).
\end{algorithm}

\section{Numerical experiments}\label{sec:numexp}
In this section we illustrate the practical performance of the proposed methods with numerical experiments. Our code is based on the C\texttt{++} library DUNE~\cite{alugrid2016, duneI, duneII} and we used the direct solver from UMFPACK~\cite{umfpack} for the solution of the linear systems arising in every time step of Algorithm~\ref{alg:isoflow} and Algorithm~\ref{alg:penisoflow}, respectively. 
The stopping criterion for the iterations was chosen as $\| \nabla \nabla_h d_t y_h^k \|_{L^2(\o)} \le \varepsilon_\mathrm{stop} = 1.0 \times 10^{-3}$ in all of the following examples.
The visualizations of the computed deformations were obtained with the open-source application ParaView~\cite{paraview}. We note that in the visualizations we neglect the degrees of freedom corresponding to the derivatives of the deformed surfaces and restrict ourselves to the representations obtained from the nodal displacements. In order to asses the quality of the approximations, we define the discrete isometry error
\[
\d_\mathrm{iso}[y_h^k] = \| [\nabla y_h^k]^\top \nabla y_h^k] - I_2 \|_{L_h^\infty(\o)},
\]
which measures the deviation of a numerical solution from an actual isometry. The numerical equilibrium solution, i.\,e. the final iterate of any algorithm, will be denoted with $y_h^\infty$ in the following. We remark that we cannot always expect the solutions to represent the physical reality, as large deformations might lead to self intersections of the deformed surfaces.

\subsection{Clamped rectangular bilayer plate} \label{subsec:clamprecplate}
In the first experiment we consider a rectangular plate occupying the domain $\o = (-5,5) \times (-2,2)$ in its reference configuration, which is clamped horizontally along the Dirichlet boundary $\GD = \{-5\} \times [-2,2]$, i.\,e. the boundary conditions are given by $y_\DD = [x,0]^\top$ and $\phi_\DD = [I_2,0]^\top$. The curvature parameter is chosen as $\alpha = 2.5$.
It has been shown in~\cite{schmidt2007} that the minimizer of the continuous energy for this problem is given by a cylinder of height $4$ (the length of the clamped boundary) and radius $\alpha^{-1}$ with an energy of $125$.
For the computation of discrete solutions, we employ two triangulations $\Th$ and $\Th^\mathrm{sym}$ of the domain $\o$, both consisting of halved squares of side-length $h = 2^{-3}$ amounting to 5120 triangles in each triangulation. The triangles in $\Th$ only have two orientations, resulting in a nonsymmetrical grid, whereas the triangles in $\Th^\mathrm{sym}$ are arranged in a symmetrical fashion as it results from dividing squares of side length $2h$ along their lines of symmetry, cf. Figure~\ref{fig:rect-evolution}.

\begin{figure}
  \centering
  \begin{subfigure}[t!]{.48\textwidth}
    \centering
    \includegraphics[width=.95\linewidth]{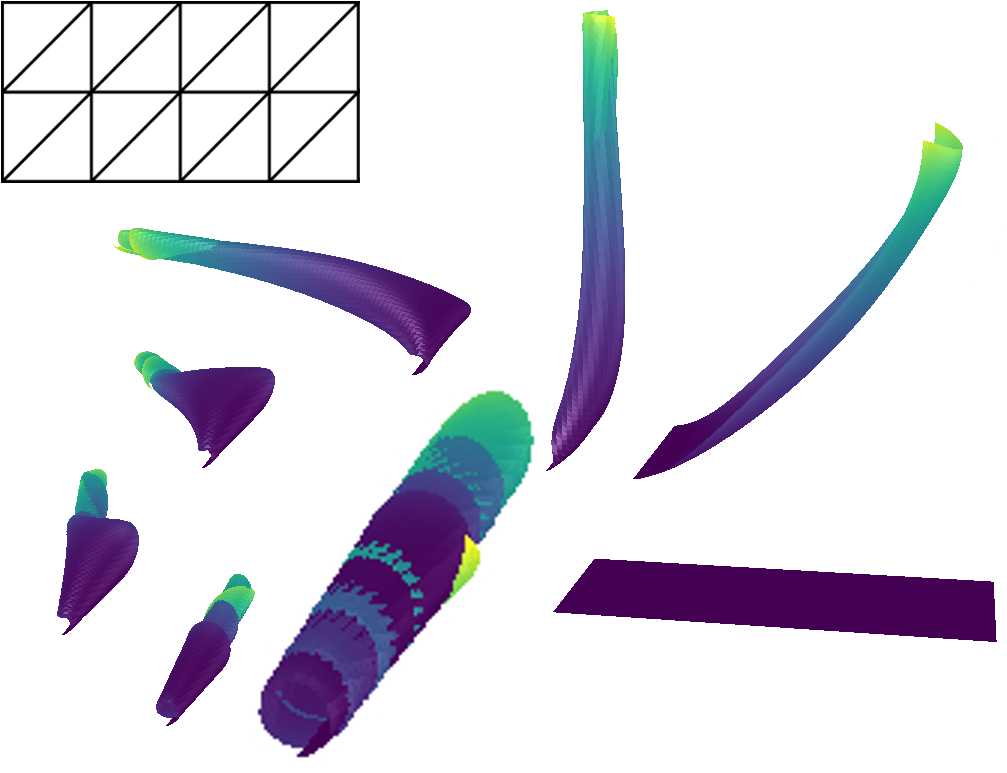}
    \caption{Nonsymmetric triangulation; snapshots of the discrete solution after (anticlockwise) 0, 20, 1.2k, 18k, 50k, 57k, 80k and 181,218 iterations of Algorithm~\ref{alg:isoflow}.}
  \end{subfigure}
  \hspace{8pt}
  \begin{subfigure}[t!]{.48\textwidth}
    \centering
    \includegraphics[width=.95\linewidth]{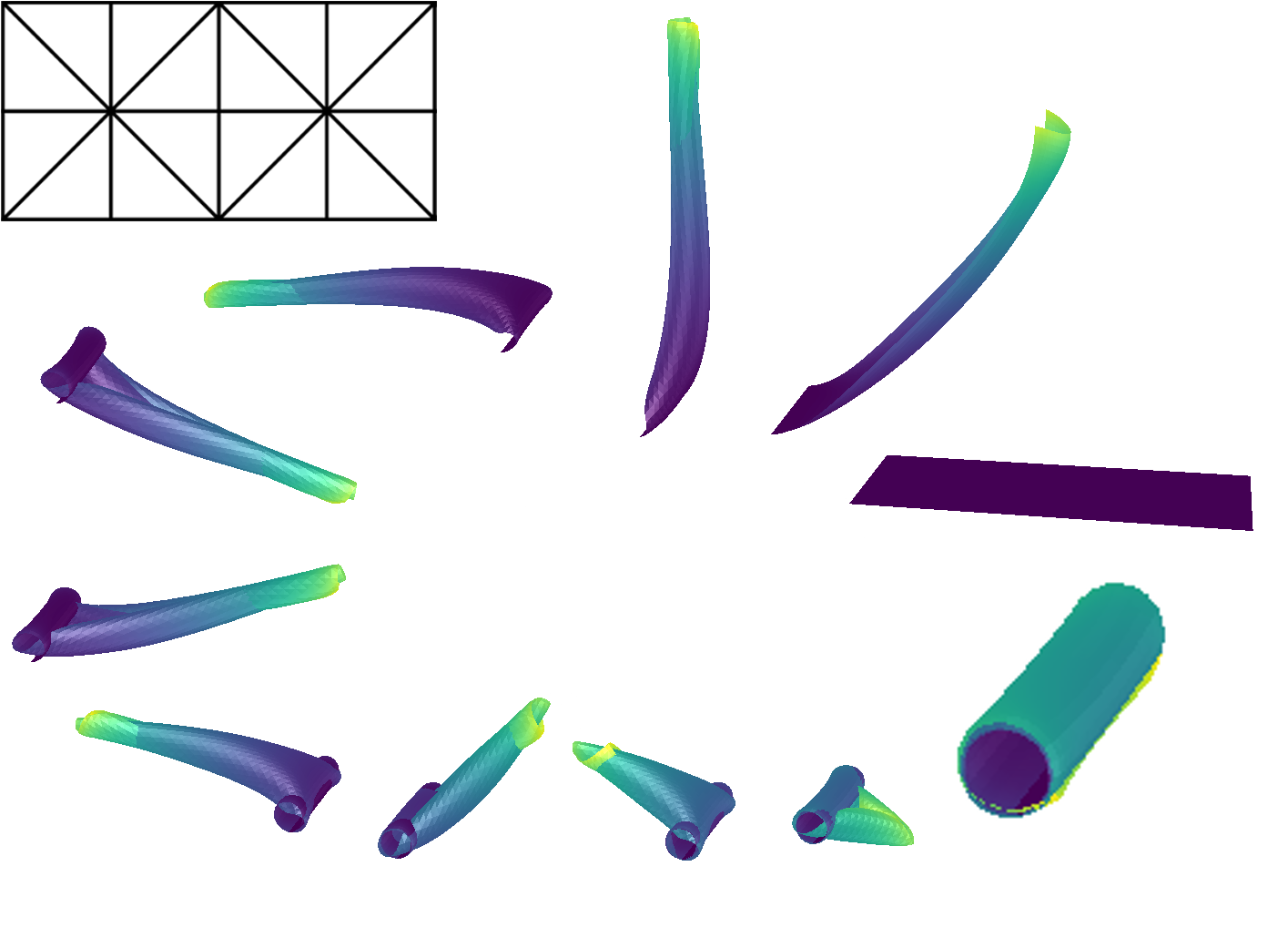}
    \caption{Symmetric triangulation; snapshots of the discrete solution after (anticlockwise) 0, 20, 1.2k, 18k, 50k, 57k, 80k, 100k, 105k, 110k and 176,629 iterations of Algorithm~\ref{alg:isoflow}.}
  \end{subfigure}
  \caption{Evolution of the discrete solutions in Section~\ref{subsec:clamprecplate}. The coloring of the surfaces corresponds to the magnitude of $|[\nabla y_h^k]^\top \nabla y_h^k] - I_2|$.}
  \label{fig:rect-evolution}
\end{figure}
\begin{figure}
  \begin{subfigure}{.49\textwidth}
    \centering
    \includegraphics[width=\linewidth]{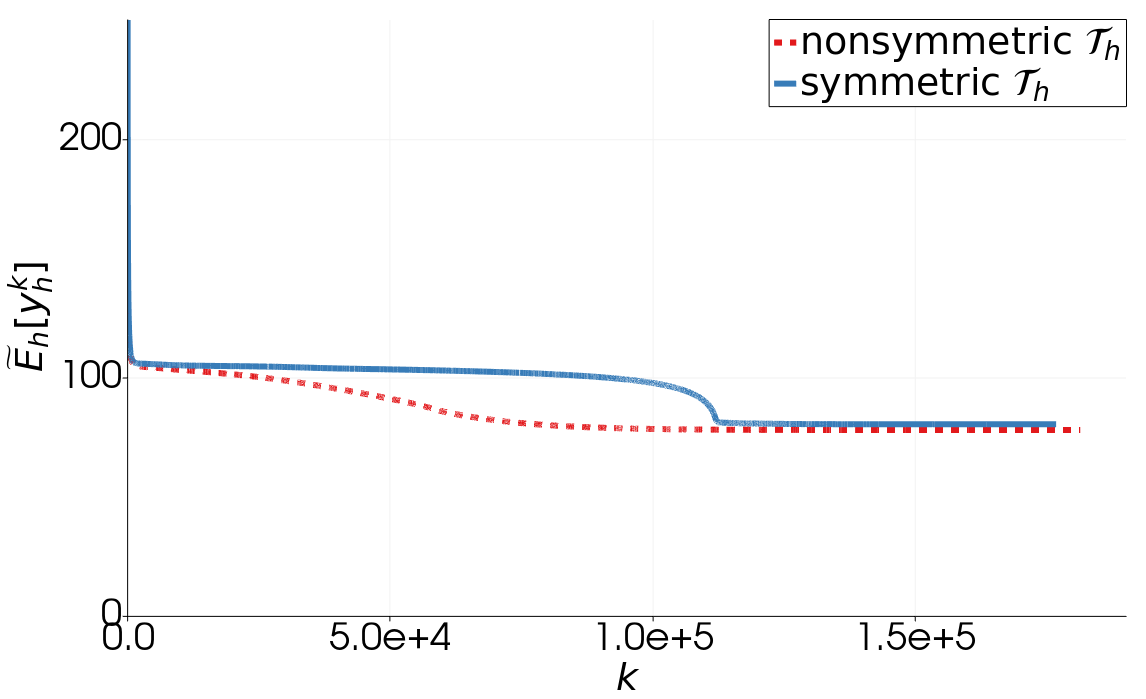}
  \end{subfigure}
  \begin{subfigure}{.49\textwidth}
    \centering
    \includegraphics[width=\linewidth]{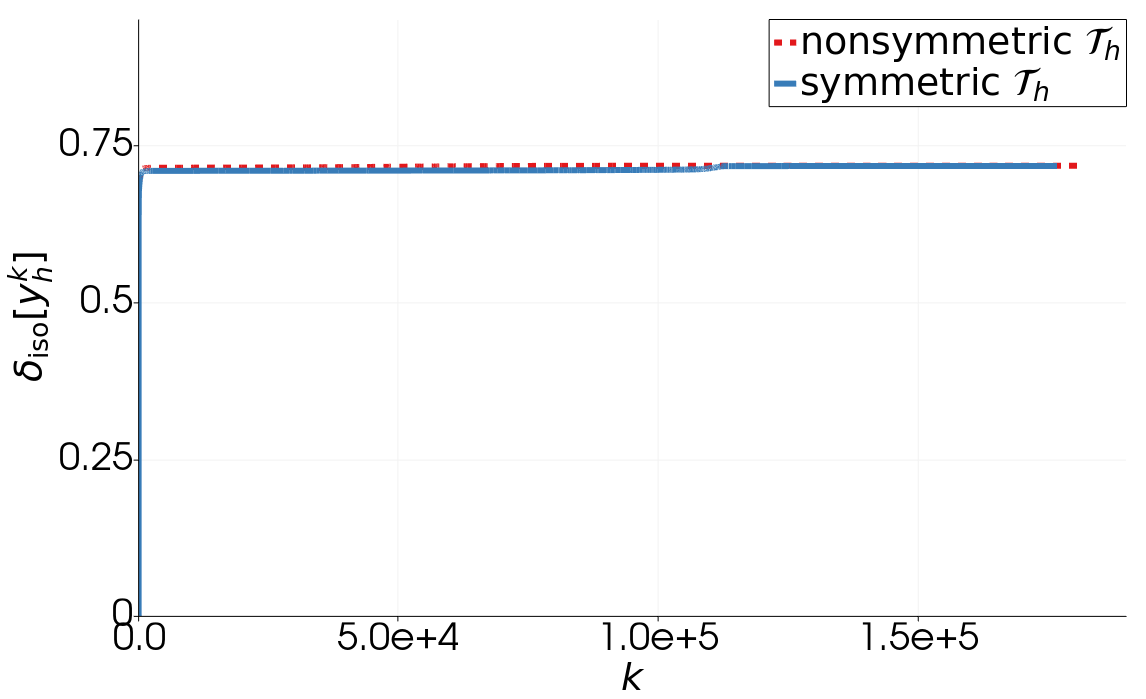}
  \end{subfigure}
  \caption{Energy decay (left) and limited growth (right) of the isometry error for the iterates of Algorithm~\ref{alg:isoflow} in Section~\ref{subsec:clamprecplate} using different triangulations. A significant decrease in the energy can be observed for the symmetrical triangulation after about 100k iterations, when the initially formed \emph{dog-ears} finally unfold and the plate coils into its cylindrical equilibrium shape.}
  \label{fig:rect-energy}
\end{figure}

In the realization of Algorithm~\ref{alg:isoflow} we chose the step size $\tau = h / 5$. The numerical equilibrium shape was reached after $181,218$ iterations in the case of the triangulation $\Th$ and after $176,629$ iterations in the case of the symmetrical triangulation $\Th^\mathrm{sym}$ with final energies of $78.060$ and $80.461$, respectively. 
The evolutions of the discrete solutions are depicted in Figure~\ref{fig:rect-evolution}. Notably, the shape of the final iterate bears a closer resemblance to a cylinder for the symmetrical triangulation, although the associated discrete energy is a bit higher. The decay of the discrete energies as well as the behaviour of the discrete isometry error are illustrated in Figure~\ref{fig:rect-energy} and support the assertions of Theorem~\ref{thm:iteration}. We note that the dependence of approximations on the triangulation disappears as meshes become finer.

\subsection{Corner clamped O-shaped plate}\label{subsec:oshape}
\begin{figure}
  \begin{subfigure}{.27\textwidth}
    \centering
    \includegraphics[width=.95\linewidth]{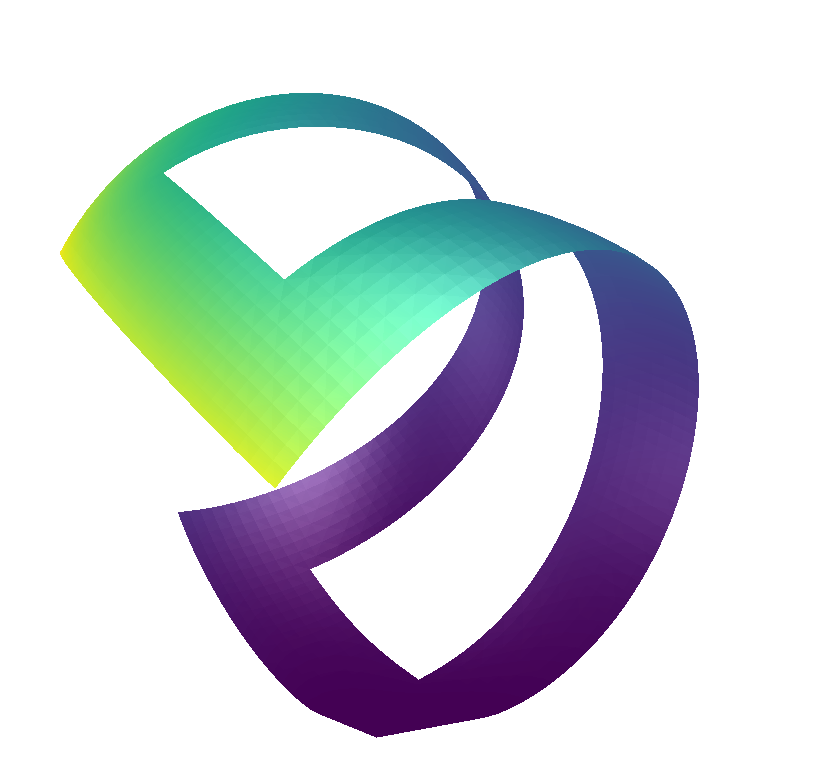}
  \end{subfigure}
  \begin{subfigure}{.27\textwidth}
    \centering
    \includegraphics[width=.95\linewidth]{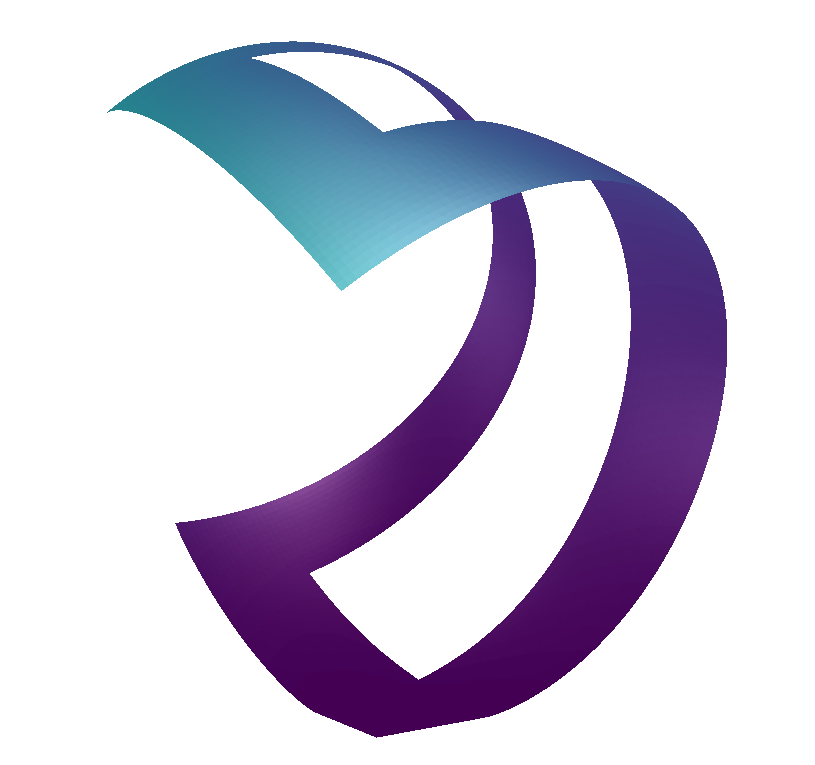}
  \end{subfigure}
  \begin{subfigure}{.27\textwidth}
    \centering
    \includegraphics[width=.95\linewidth]{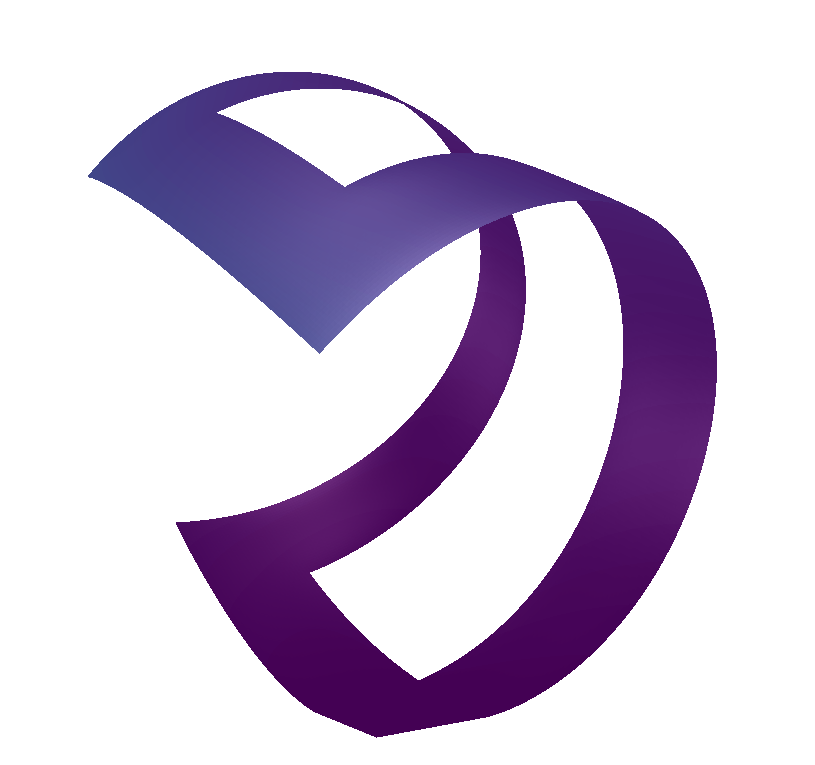}
  \end{subfigure}
  \begin{subfigure}{.12\textwidth}
    \centering
    \includegraphics[width=.95\linewidth]{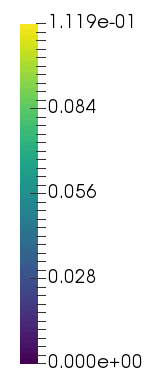}
  \end{subfigure}
  \caption{Numerical equilibrium states of the O-shaped plate in Section~\ref{subsec:oshape} for the triangulations $\Tl$, $\ell = 3,4,5$ with corresponding energies of $0.8869$, $1.444$ and $2.026$ (left to right) colored by the magnitude of $|[\nabla y_h^k]^\top \nabla y_h^k] - I_2|$.}
  \label{fig:oshape-equi}
\end{figure}

\begin{figure}
  \begin{subfigure}{.42\textwidth}
  \centering
  \begin{tabular}{r r r r}\hline
    $\ell$ & \#iter & $\tE_h[y_h^\infty]$ & $\d_\mathrm{iso}[y_h^\infty]$ \\\hline
    1 & 1922  & \num{-2.813e-1} & \num{5.181e-1} \\
    2 & 2829  & \num{4.133e-1}  & \num{2.388e-1} \\
    3 & 4513  & \num{8.869e-1}  & \num{1.119e-1} \\
    4 & 8589  & \num{1.444}     & \num{5.247e-2} \\
    5 & 20005 & \num{2.026}     & \num{2.363e-2} \\\hline
  \end{tabular}
  \end{subfigure}
  \hspace{10pt}
  \begin{subfigure}{.42\textwidth}
  \centering
  \begin{tabular}{r r r r}\hline
    $k$ & \#iter & $\tE_h[y_h^\infty]$ & $\d_\mathrm{iso}[y_h^\infty]$ \\\hline
    3  & 1283  & \num{9.332e-1} & \num{4.339e-1}  \\
    2  & 2318  & \num{1.228}    & \num{2.101e-1}  \\
    1  & 4407  & \num{1.370}    & \num{1.048e-1}  \\
    0  & 8589  & \num{1.444}    & \num{5.247e-2}  \\
    -1 & 16956 & \num{1.481}    & \num{2.625e-2}  \\\hline
  \end{tabular}
  \end{subfigure}
  \caption{Dependence on mesh size $h = 2^{-\ell}$ (left, $\tau =2^{-\ell}/5$) and step size $\tau = 2^{-\ell+k}/5$ (right, $h = 2^{-4}$) of iteration numbers, final energies and isometry error of the final iterate for the O-shaped plate in Section~\ref{subsec:oshape} for different triangulations and step sizes.}
  \label{fig:oshape-tab}
\end{figure}
In the second example we consider an O-shaped plate given by the domain $\o = (-5,5) \times (-2,2) \setminus [-4,4] \times [-1,1]$ which is horizontally clamped along the corner $\GD = \{-5\} \times [-2,-1] \cup [-5,-4] \times \{-2\}$, i.\,e. the boundary conditions are given by $y_\DD = [x,0]^\top$ and $\phi_\DD = [I_2,0]^\top$.
We first compute the resulting discrete deformations for the curvature parameter $\a = 0.5$ on a sequence $\Thsym = \Tlsym$ of triangulations, which consist of halved squares with side length $h = 2^{-\ell}$ resulting from the division of squares with side length $2h$ along their lines of symmetry. The finest triangulation consists of $49152$ triangles and the resulting solutions are defined by $227511$ degrees of freedom. For each triangulation the step size in Algorithm~\ref{alg:isoflow} was chosen as $\tau = h/5$. The resulting numerical equilibrium states for $\ell = 3,4,5$ are depicted in Figure~\ref{fig:oshape-equi}.

In another sequence of computations, we computed the discrete solutions on the triangulation $\ell = 4$ with step sizes $2^k \times (h/5)$, $k = 3,2,1,0,-1$. The iteration numbers, final energies and isometry errors for the computations are shown in the tables in Figure~\ref{fig:oshape-tab}. As expected, the isometry error depends linearly on the step size $\tau$. The same effect can be observed for the number of iterations that are necessary to reach the numerical equilibrium state.

\subsection{Transition between contact effects} \label{subsec:obst}
We consider the O-shaped plate from Section~\ref{subsec:oshape} which is given in its reference configuration by the domain $\o = (-5,5) \times (-2,2) \setminus [-4,4] \times [-1,1]$. 
As before, we assume that it is horizontally clamped along the corner $\GD = \{-5\} \times [-2,-1] \cup [-5,-4] \times \{-2\}$, i.\,e. the boundary conditions are given by $y_\DD = [x,0]^\top$ and $\phi_\DD = [I_2,0]^\top$. 
We assume an impenetrable obstacle to be placed at height $x_3 = 1$ and compute approximate solutions for different values of the curvature parameter $\alpha$ and without spontaneous curvature (i.\,e. $\alpha = 0$) in the presence of a body force $f = (0,0,c_f)^\top$ for the penalized energy~\eqref{eq:pen-discenergy} with Algorithm~\ref{alg:penisoflow}. 
We employ the triangulation $\mathcal{T}_3$ consisting of halved squares with side length $h = 2^{-3}$ and choose the step size $\tau = h / 50$ and the stopping criterion $\varepsilon_\mathrm{stop} = 10^{-3}$. Figure~\ref{fig:obst} shows the computed equilibrium states for the penalty parameter $\varepsilon = 1.25 \times 10^{-1}$ with $c_f = 2 \times 10^{-3}$, $c_f = 1 \times 10^{-2}$ and $c_f = 2 \times 10^{-2}$ and with $\alpha = 0.25$, $0.5$ and $0.75$, respectively. 

\begin{figure}
  \begin{subfigure}{.49\textwidth}
    \centering
    \includegraphics[width=\linewidth]{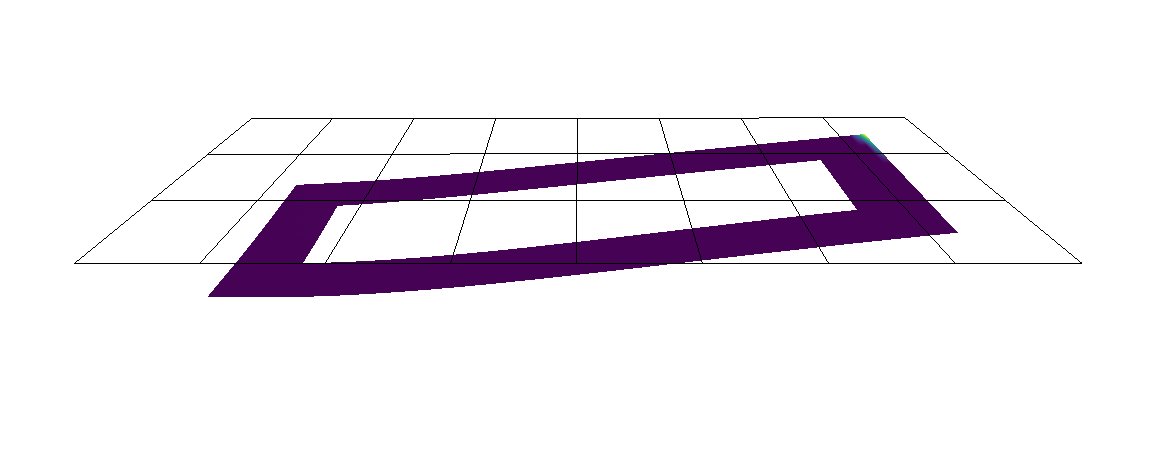}
  \end{subfigure}
  \begin{subfigure}{.49\textwidth}
    \centering
    \includegraphics[width=\linewidth]{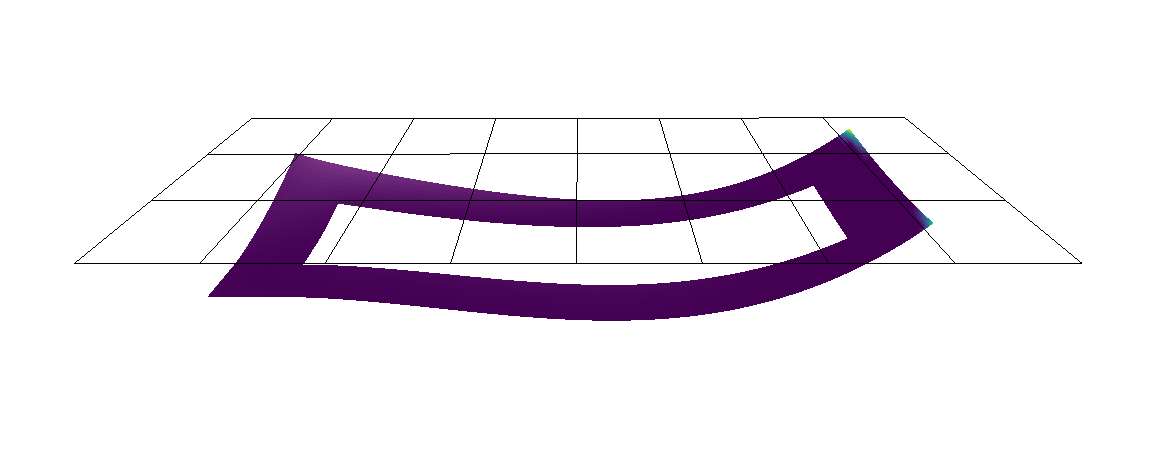}
  \end{subfigure}
  \begin{subfigure}{.49\textwidth}
    \centering
    \includegraphics[width=\linewidth]{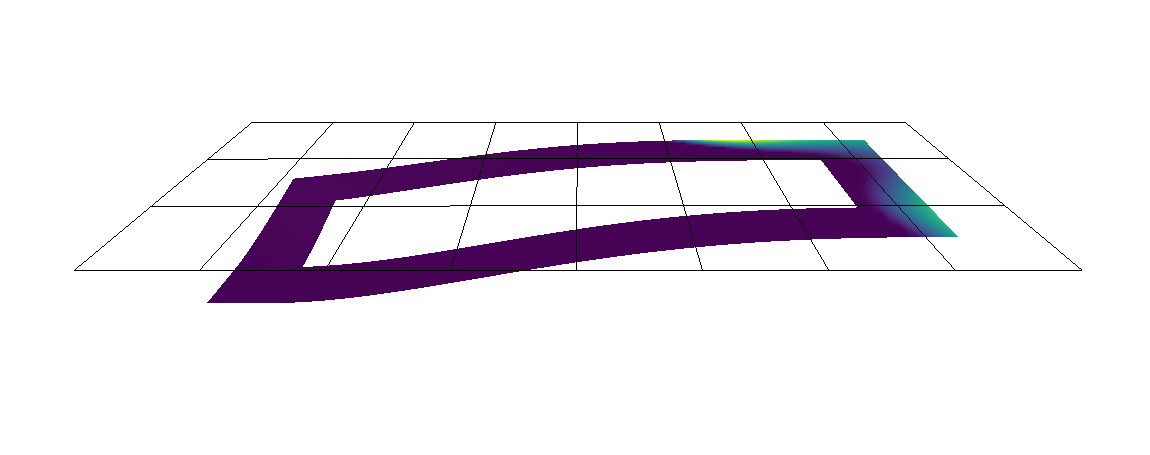}
  \end{subfigure}
  \begin{subfigure}{.49\textwidth}
    \centering
    \includegraphics[width=\linewidth]{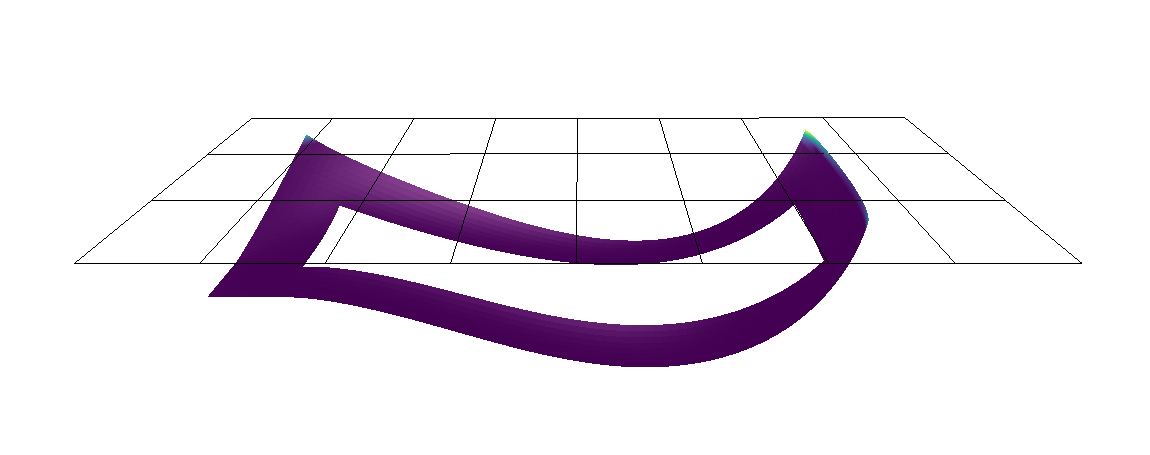}
  \end{subfigure}
  \begin{subfigure}{.49\textwidth}
    \centering
    \includegraphics[width=\linewidth]{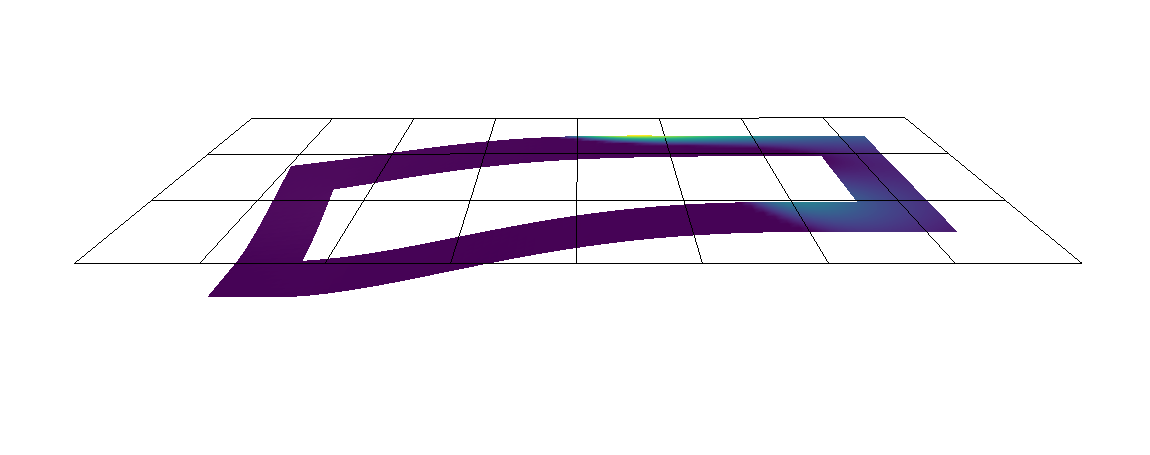}
  \end{subfigure}
  \begin{subfigure}{.49\textwidth}
    \centering
    \includegraphics[width=\linewidth]{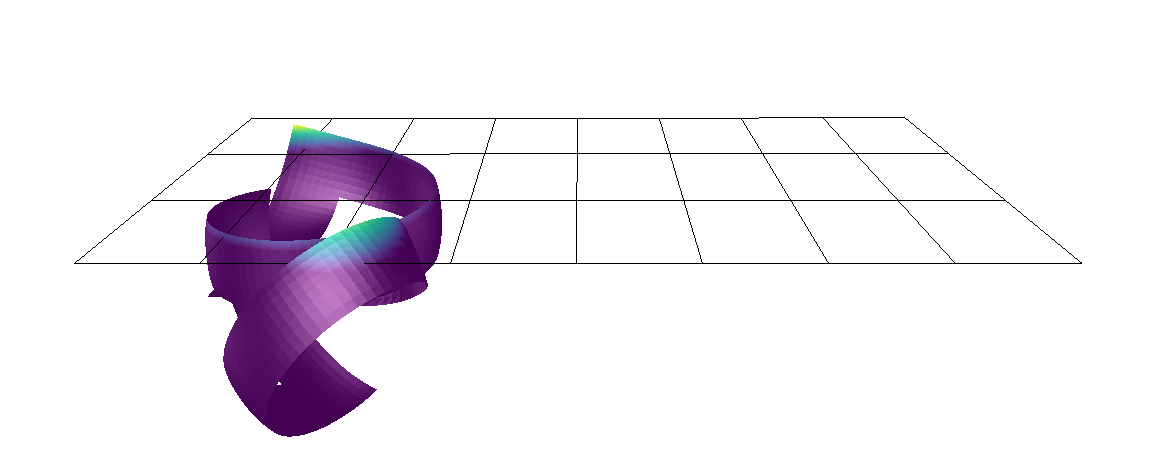}
  \end{subfigure}
  \caption{Numerical equilibrium states for the obstacle problem in Section~\ref{subsec:obst} with body force $c_f = 2 \times 10^{-3}$, $1 \times 10^{-2}$ and $2 \times 10^{-2}$ (left) and curvature parameter $\alpha = 0.25$, $0.5$ and $0.75$ (right), colored by the obstacle penetration $(y_{h,3}-1)_+$.}
  \label{fig:obst}
\end{figure}

For the smallest considered values of both $c_f$ and $\alpha$, the obstacle penetration is largest at the corner diametrically opposed to the clamped corner. In the case of spontaneous curvature, we observe only isolated contact points and no tendency towards the formation of a contact region for higher values of the curvature parameter $\alpha$. Instead, for higher values of $\alpha$, the plate curls into a ball-like shape underneath the obstacle where severe self-intersections can be observed. 

\begin{figure}
  \begin{subfigure}{.49\textwidth}
    \centering
    \includegraphics[width=\linewidth]{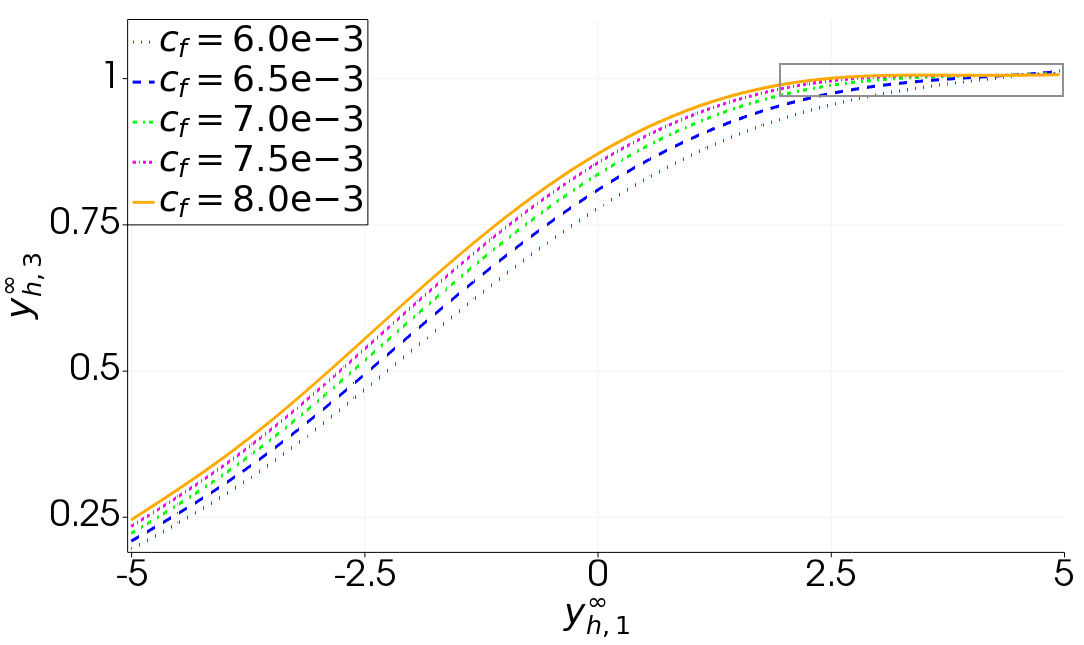}
  \end{subfigure}
  \begin{subfigure}{.49\textwidth}
    \centering
    \includegraphics[width=\linewidth]{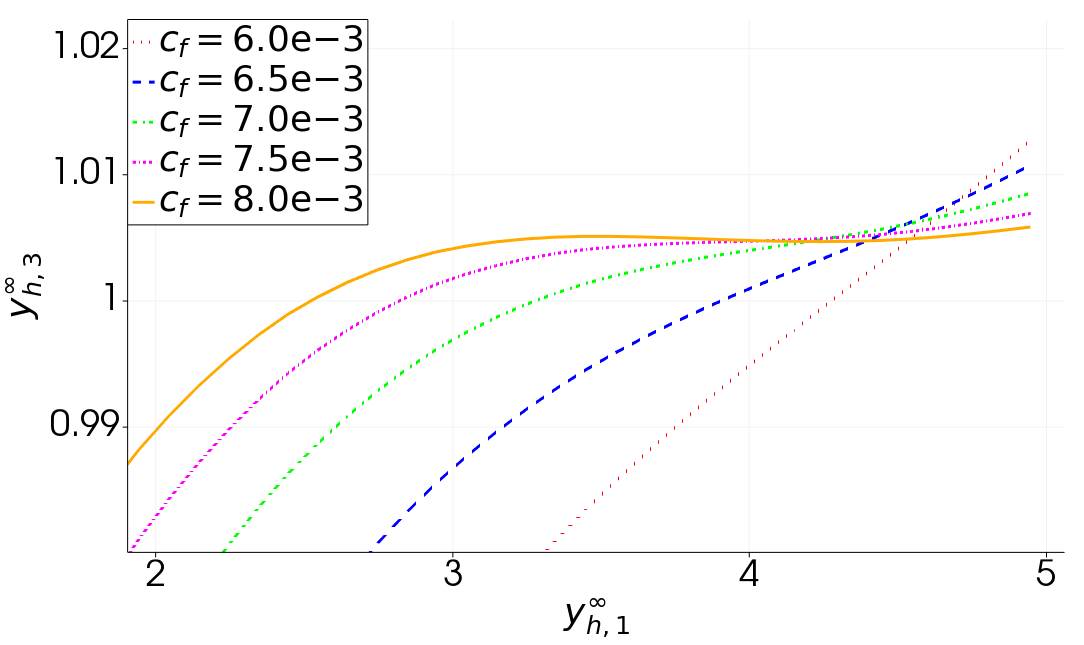}
  \end{subfigure}
  \caption{Plot of the $x$- and $z$-components of the displacement field for different magnitudes $c_f$ of the body force along the rear edge $[-5,5] \times \{2\}$ of the plate (left) and magnification of the critical part marked by the gray box, where the obstacle contact occurs (right).}
  \label{fig:contact}
\end{figure}

\begin{figure}
  \vspace{25pt}
  \centering
  \includegraphics[width=.49\linewidth]{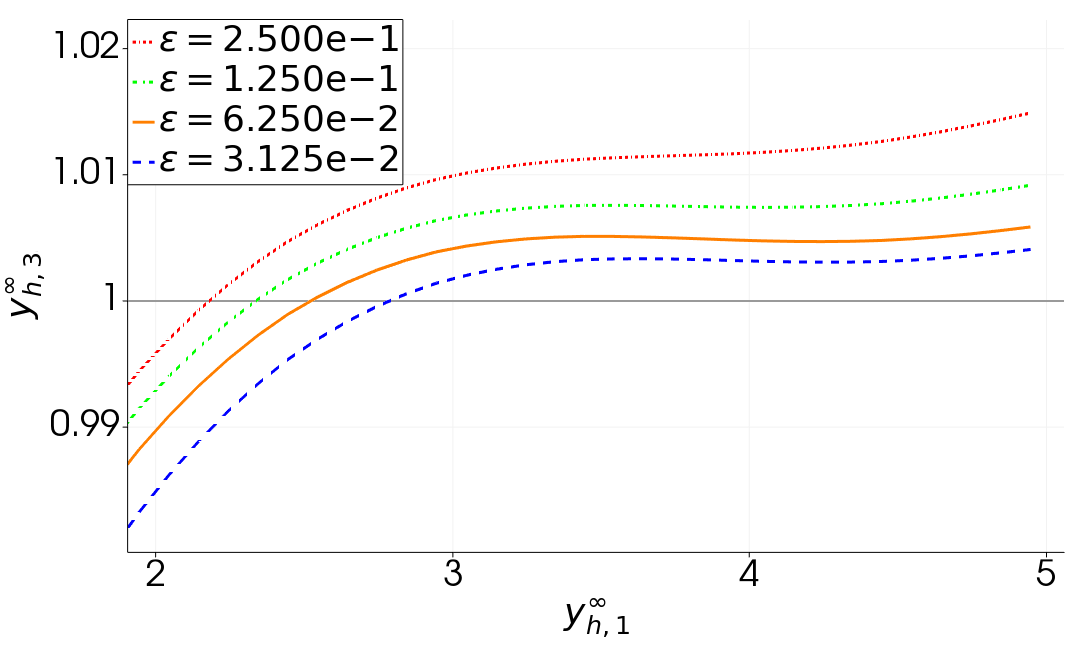}
  \caption{Influence of the penalty parameter $\varepsilon$ on the critical part of the deformed boundary side $[-5,5] \times \{2\}$ where obstacle penetration occurs for the body-force magnitude $c_f = 8.0 \times 10^{-3}$.}
  \label{fig:contact-eps}
\end{figure}

In the absence of spontaneous curvature (i.\,e. in the case $\alpha = 0$), our experimental findings indicate a change of qualitative properties of the contact region as the body force changes.
For small body forces, the mapping $x \mapsto y_h^\infty(x,2)$ (i.\,e. the rear edge of the deformed plate) is convexly shaped and the numerical solution implies that there is only one point of contact in the physical solution of the obstacle problem. For large body forces we experimentally observe the presence of inflection points along the $x$-axis indicating that in the physical solution one  could expect a contact region with positive surface measure. The change from contact point to contact region is further investigated in Figure~\ref{fig:contact}, where we plotted the $x$- and $z$-components of the equilibrium displacement field along the rear edge $[-5,5] \times \{2\}$ of the plate for the penalty parameter $\varepsilon = 6.25 \times 10^{-2}$ and several values of $c_f$. 

In Figure~\ref{fig:contact-eps} we plotted the $x$- and $z$-components of the final iterate along the critical part of the rear edge that were obtained for $c_f = 8.0 \times 10^{-3}$ with several penalization parameters $\varepsilon > 0$. The results of the simulations suggest that the change from contact point to contact region happens somewhere between $c_f = 7.0 \times 10^{-3}$ and $c_f = 8.0 \times 10^{-3}$.

The table in Figure~\ref{fig:obst-table} shows the iteration numbers, final energies $\tEpenh[y_h^\infty]$ and obstacle penetrations 
\[
\d_\mathrm{pen}[y_h^\infty] = \|(y^\infty_{h,3} - 1)_+\|_{L_h^\infty(\o)}
\]
for several values of $c_f$ and $\varepsilon$. The isometry errors are negligible in these computations (we have $\d_\mathrm{iso}[y_h^\infty] \le 1.070 \times 10^{-5}$ in all examples) , since the obstacle prevents large deformations leading to artificial in-plane stretching. The obstacle penetration $\d_\mathrm{pen}[y_h^\infty]$ seems to decrease with a higher order than $\varepsilon^{1/4}$ as indicated by Corollary~\ref{cor:penest}. We note that an appropriate coupling of the penalty parameter and the step size appears necessary to guarantee a consistency property of the numerical equilibrium state.
The reason for this lies in the semi-implicit treatment of the penalty functional. In our example, the body force contribution and the difference between the implicitly treated convex part and the explicitly treated concave part might cancel each other out in the discrete gradient flow, which corresponds to a lack of consistency in the discretization of the continuous flow. The step size should thus be chosen according to $\tau \le C c_f \varepsilon$ with a constant $C$ depending on the triangulation~$\Th$.

\begin{figure}
  \vspace{25pt}
  \centering
  \begin{tabular}{c c r r r r}\hline
    $c_f$ & $\varepsilon$ & \#iter & $\tEpenh[y_h^\infty]$ & $\d_\mathrm{pen}[y_h^\infty]$ \\ \hline
    \num{6.0e-3} & \num{5.0e-1}   & 5121  & \num{-6.844e-2} & \num{3.486e-2} \\
    \num{6.0e-3} & \num{2.5e-1}   & 8271  & \num{-6.821e-2} & \num{2.405e-2} \\
    \num{6.0e-3} & \num{1.25e-1}  & 14029 & \num{-6.804e-2} & \num{1.711e-2} \\
    \num{6.0e-3} & \num{6.25e-2}  & 24200 & \num{-6.788e-2} & \num{1.278e-2} \\\hline
    \num{8.0e-3} & \num{2.5e-1}   & 6773  & \num{-9.749e-2} & \num{1.483e-2} \\
    \num{8.0e-3} & \num{1.25e-1}  & 11397 & \num{-9.733e-2} & \num{9.114e-3} \\
    \num{8.0e-3} & \num{6.25e-2}  & 20034 & \num{-9.721e-2} & \num{5.818e-3} \\
    \num{8.0e-3} & \num{3.125e-2} & 35801 & \num{-9.708e-2} & \num{4.188e-3} \\\hline
  \end{tabular}
  \caption{Iteration numbers, final energies and obstacle penetrations of the final iterates for different penalty parameters and magnitudes of the body force in the example in Section~\ref{subsec:obst}.}
  \label{fig:obst-table}
\end{figure}

\bibliographystyle{siam}
\bibliography{references}

\end{document}